\documentclass{amsart}

\usepackage{url}

\newcommand{\gambar}{\bar{\gamma}}
\newcommand{\cansec}{\vartheta}
\newcommand{\atp}{f_{2g+1}}

\newcommand{\Q}{\mathbb{Q}}

\newcommand{\Z}{\mathbb{Z}}
\newcommand{\PP}{\mathbb{P}}
\newcommand{\Zip}{\Z[1/p]}
\DeclareMathOperator{\Res}{Res}
\newcommand{\Qp}{\Q_p}
\newcommand{\Zp}{\Z_p}
\newcommand{\F}{\mathbb{F}}
\newcommand{\Cp}{\mathbb{C}_p}
\newcommand{\Qpb}{\bar{\Q}_p}

\newcommand{\dr}{\textup{dR}}
\newcommand{\hdr}{H_{\dr}}

\newcommand{\OO}{\mathcal{O}}
\newcommand{\tang}{\mathcal{T}}

\newcommand{\LL}{\mathcal{L}}

\newcommand{\loti}{\log_{\tang}}
\newcommand{\hoti}{h_{\tang}}

\newcommand{\XX}{\mathcal{X}}
\newcommand{\UU}{\mathcal{U}}
\newcommand{\KK}{\mathcal{K}}
\newcommand{\bom}{\bar{\omega}}
\newcommand{\Bom}{\pmb{\omega}}

\newcommand{\Div}{\operatorname{Div}}
\newcommand{\Spec}{\operatorname{Spec}}
\newcommand{\supp}{\operatorname{supp}}
\renewcommand{\div}{\operatorname{div}}
\newcommand{\tr}{\operatorname{tr}}
\newcommand{\hloc}{h_{\textup{loc}}}
\newcommand{\hloco}[1]{\hloc(#1)_\omega}

\newcommand{\frev}{f^{\textup{rev}}}

\theoremstyle{plain}

\newtheorem{theorem}{Theorem}[section]

\newtheorem{proposition}[theorem]{Proposition}
 
\newtheorem{lemma}[theorem]{Lemma}
\newtheorem{corollary}[theorem]{Corollary}
\theoremstyle{definition}

\newtheorem{remark}[theorem]{Remark}

\numberwithin{equation}{section}

\begin{document}
\dedicatory{Dedicated to the memory of Robert F. Coleman (1954-2014)}

\title[Quadratic Chabauty]{Quadratic Chabauty: $p$-adic heights and integral points on hyperelliptic curves}

\author{Jennifer S. Balakrishnan}
\address{Department of Mathematics, Harvard University, 1 Oxford Street, Cambridge, MA 02138, U.S.A. and Mathematical Institute, University of Oxford, Andrew Wiles Building, Radcliffe Observatory Quarter, Woodstock Road, Oxford, OX2 6GG, United Kingdom}\email{balakrishnan@maths.ox.ac.uk}

\author{Amnon Besser}
\address{Mathematical Institute, University of Oxford, Andrew Wiles Building, Radcliffe Observatory Quarter, Woodstock Road, Oxford, OX2 6GG, United Kingdom and Department of Mathematics, Ben-Gurion University of the Negev, P.O.B. 653,Be'er-Sheva 84105 Israel}\email{bessera@math.bgu.ac.il}

\author{J. Steffen M\"{u}ller}
\address{Fachbereich Mathematik, Universit\"{a}t Hamburg, Bundesstrasse 55,
20146 Hamburg, Germany and Institut f\"ur Mathematik, Carl von Ossietzky
Universit\"{a}t Oldenburg, 26111 Oldenburg, Germany}\email{jansteffenmueller@gmail.com}

\begin{abstract}
  We give a formula
  for the component at $p$ of the $p$-adic height pairing of a divisor
  of degree $0$ on a hyperelliptic curve. We use this to give a
  Chabauty-like method for finding $p$-adic approximations to $p$-integral points on such curves
  when the
  Mordell-Weil rank of the Jacobian equals the genus. In this case we
  get an explicit bound for the number of such $p$-integral points, and we
  are able to use the method in explicit computation. An important
  aspect of the method is that it only requires a basis of the
  Mordell-Weil group tensored with $\Q$.
\end{abstract}
%\subjclass[2010]{Primary 11S80, 14G40, 11Y50; Secondary 11G30, 11D41}

\maketitle

\section{Introduction}
\label{sec:intro}

Chabauty's method~\cite{Chab41}, later made effective by
Coleman~\cite{Col85a}, is a fantastic tool for bounding the number of
rational points of a curve $X$ over a number field $F$ and for finding
$p$-adic approximations to these points, when the genus $g$ is larger
than the Mordell-Weil rank of the Jacobian $J$ of $X$ over $F$. This method
has been put to effective use in many instances over the last thirty
years; for examples illustrating this technique, see~\cite{mccallum-poonen:method} or \cite{flynn:flexible}.
Chabauty wrote down $p$-adic functions that vanish on the set
of $F$-rational points $X(F)$ and Coleman identified these functions
as $p$-adic Coleman integrals of holomorphic forms~\cite{Col85}.

The fascinating recent work of Kim~\cite{Kim05,Kim09,Kim10a,Kim-Coa10,BCKW12}
on a \emph{non-abelian} Chabauty method
gives hope that the restriction on the rank of $J$ may be removed by
using more general iterated Coleman integrals. When Kim's method
applies and a Coleman function vanishing on $X(F)$ is found, it can
sometimes be computed explicitly by recent progress on the
computation of such
functions~\cite{BBK09,balakrishnan:iterated,Bes-Bal10}.

In a recent example of Kim's method~\cite{Kim10}, as corrected
in~\cite{BKK11}, an explicit Coleman function was given in the case of an
elliptic curve $E/\Q$ with Mordell-Weil rank $1$ over $\Q$ (satisfying some
auxiliary conditions) and shown to vanish on all integral points. This
function was furthermore explicitly computed in~\cite{BKK11} and numerically exhibited 
to vanish on integral points.

In~\cite{Bes12}, the first two named authors gave an entirely new proof
of Kim's result,
removing, on the way, some of the auxiliary assumptions by showing that
the function identified by Kim was essentially the component at $p$ of
the $p$-adic height and relying on the quadraticity of this height as
a function on $J$ and on the description of the other components as
intersection multiplicities.

The goal of the present work is to extend the methods of~\cite{Bes12}
to hyperelliptic curves. Suppose that $X$ is a hyperelliptic curve
over $\Q$ given by the affine
equation 
\begin{equation}\label{eq:Weierstrass}
  y^2=f(x)\;,
\end{equation}
with $f$ a polynomial of degree $2g+1$ over
$\Z$ which does not reduce to a square modulo any prime number.
Let $p$ be a prime, let $X_p=X\otimes \Q_p$, and let $J$ be the Jacobian of $X$.
We can prove
\begin{theorem}\label{introthm}
  If the Mordell-Weil rank of $J$ over $\Q$ is exactly $g$, then
  there exists a Coleman
  function $\rho$ on $X_p$ and a finite set of values $T$ such that
  $\rho(\UU(\Zip))\subset T$, where $\UU(\Zip)$ is the set of
  $p$-integral solutions to~\eqref{eq:Weierstrass}.  If $X$ has good reduction
  at $p$, then $T$ is effectively computable 
  and $\rho$ is effectively computable from a basis for $J(\Q)
  \otimes \Q$.
\end{theorem}
A more precise version of this result will be given in
Theorem~\ref{cutting-function} in Section~\ref{sec:integral}. We
use this to give an effective bound on the number of integral points
in Theorem~\ref{boundthm}.

We expect that the result above will lead to a practical method
for recovering all integral points on hyperelliptic curves satisfying the
assumptions of the theorem (see Remark~\ref{practical}). In Example~\ref{ex:g2}, we show
how to carry this out in practice and find all integral points with $x$-coordinate having
absolute value less than a prescribed bound.

A different algorithm for the computation of integral points on
hyperelliptic curves is given in~\cite{BMSST08}.
Their algorithm combines linear forms in logarithms with a modified version of
the Mordell-Weil sieve and does not assume that the rank of $J(\Q)$ is equal to
the genus.
However, in contrast to our method, it does require generators of the free part of $J(\Q)$ which limits its
applicability, since currently there is no practical algorithm for the
saturation of a finite index subgroup of $J(\Q)$ unless $g = 2$
\cite{stoll:height_constantII}; the case $g=3$ is currently being worked out \cite{stoll:g3}.
For a recent approach to this problem using arithmetic intersection theory, see
\cite{holmes}.

The paper~\cite{BMSST08} contains (see page~3) a brief discussion of other possible attacks on the problem using $S$-units and the Baker or Skolem methods (see~\cite{Bil-Han98} and~\cite{Sma98}) and why these often fail in practice.

Our proof relies on the description of the $p$-adic height pairing
given by Coleman and Gross~\cite{Col-Gro89}, the quadraticity of the pairing, and
$p$-adic Arakelov 
theory~\cite{Bes00}. For any $X$ and $F$, the height pairing
\begin{equation*}
  h: J(F) \times J(F) \to \Qp\;,
\end{equation*}
depending on some auxiliary data, is defined initially for divisors $D_1$
and $D_2$ of degree zero 
with disjoint support as a sum of local height pairings
\begin{equation*}
  h(D_1,D_2) = \sum_v h_v(D_1,D_2)
\end{equation*}
over the finite places $v$ of $F$. As in~\cite{Bes12}, a key point is
to remove the disjoint support restriction by extending the local
height pairings relative to a choice of tangent vectors, as suggested
in the classical case in~\cite{Gross86}. For our hyperelliptic curve,
we make a certain consistent choice of such tangent vectors. The
resulting local height pairing gives rise to a function $\tau(x)=
h_p(x-\infty,x-\infty)$ which is computed explicitly as a Coleman
integral in Theorem~\ref{tauthm} and is the main summand for the
function $\rho$ in Theorem~\ref{introthm}. The finite set of values $T$ results from the sum of
the height pairings $h_q$ for $q\ne p$ and is discussed in detail in
Proposition~\ref{locaway}.

Let us explain the reason for the restriction that $X$ is assumed to have good reduction at
$p$. As our work ultimately relies on $p$-adic Arakelov
theory~\cite{Bes00}, which
has been developed using Vologodsky's integration
theory~\cite{Vol01}, it is insensitive to the type of reduction
at $p$. However, Vologodsky's integration has been, so far, difficult to
compute, and the computation of Coleman integrals has been, thus far, done in
the good reduction case. The situation will likely change
with~\cite{Bes-Zer13}, which will reduce Vologodsky integration on
semi-stable curves to Coleman integration.

We briefly discuss possible applications (see also
Remark~\ref{higherfields}). The function $\rho$ is
given, on the residue disk of $p$-adic points reducing to a given
point, by a convergent power series. It is therefore possible, in
reasonable time, to compute all solutions of $\rho(x)\in T$ to high
$p$-adic precision. To find the integral points, one needs to (provably)
identify which solutions of $\rho(x)\in T$ do come from integral
points and recover those points. This can be done by combining the
techniques of this paper with the Mordell-Weil sieve. We will report
on this in future work.

It is furthermore easy, given some initial coefficients in the power
series expansions making up $\rho$, to deduce a bound on the total
number of possible solutions to $\rho(x)\in T$, hence on the number of
$p$-integral solutions to~\eqref{eq:Weierstrass}. 
A fairly crude version of this is given in Theorem~\ref{boundthm}. 
The bound depends on some computations, including those of Coleman integrals, on the curve. 
One might hope for a
bound which, like~\cite[(ii) on p.~765]{Col85a}, will depend only on some simple numerical
data of the curve, such as genus and types of bad reduction, but this
is unfortunately not obvious at this point in time.

We conclude by giving a number of numerical examples to illustrate our techniques.

\section*{Acknowledgements}
We would like to thank Michael Stoll for computing all integral points on the
curve $X$ in Example~\ref{ex:g2}.  Moreover, we would like to thank William Stein and
NSF grant DMS-0821725 for access to
\texttt{mod.math.washington.edu}. The second author would like to thank the School of Mathematical
Sciences at Arizona State University, where a significant part of the
research was carried out. The first author was supported by
NSF grant DMS-1103831. The second author's stay at Oxford was funded by ERC grant 204083 and by EPSRC grant I020519/1. The third author was supported by DFG grant
KU~2359/2-1.
\section{The local height pairing as a Coleman function}
\label{sec:local}

The goal of this section is to prove Theorem~\ref{tauthm} describing
the function $\tau$, essentially the component above $p$ for the
height pairing on a hyperelliptic curve, explicitly as a Coleman function.
In this section we let $p$ be a prime number and consider a smooth and 
proper curve $C$ of genus $g$ (we will later assume $C$ is hyperelliptic) defined over the
algebraic closure $\Qpb$ of the field of $p$-adic numbers. We fix a
decomposition
\begin{equation}\label{eq:decompi}
  \hdr^1(C) = W\oplus \Omega^1(C)\;,
\end{equation}
such that $W$ is isotropic with respect to the cup product, and a branch of
the $p$-adic logarithm which we denote by $\log$ (such a branch is
uniquely determined by declaring the value of $\log(p)$). Given this data, the
second named author defined in~\cite{Bes00} the canonical (up to an
additive constant) $p$-adic
Green function $G:C\times C- \Delta \to \Qpb$, where $\Delta \subset
C\times C$ is the diagonal.

The Green function $G$
computes the local component of the height pairing at a place $v$
above $p$ as follows: Suppose $X$ is a curve over a number field $F$ and
$C$ is an extension of scalars of $X$ via $F\to F_v\subset
\Qpb$. Suppose further that $D_1$ and $D_2$ are two degree $0$ divisors with
disjoint support on $X$ and denote in the same way their preimages in
$C$, $D_1=\sum n_i P_i$, $D_2= \sum m_j Q_j$. Then, the local
component $h_v(D_1,D_2)$ of the $p$-adic height pairing is $\tr_v
\sum_{ij} n_i m_j G(P_i,Q_j)$, where the sum is in $F_v$ and the trace
map $F_v\to \Qp$ is some $\Qp$-linear map obtained from the data for
the height pairing. When $F=\Q$, it is simply the identity map, so we
may ignore it for our applications.

In almost perfect analogy with classical Arakelov theory, the Green
function $G$ corresponds to a ``metric'' on the line bundle $\OO(\Delta)$ on
$C\times C$. In $p$-adic Arakelov
theory the right analogue of the notion of a metric
corresponds more closely to the logarithm of the metric. It leads to
the notion of a \emph{log-function}, a (Coleman) function on the total space of
the bundle minus the $0$-section, which is fiber-by-fiber the
$p$-adic logarithm up to an additive constant (this is independent of
the trivialization of the fiber). There is a notion of curvature for
log-functions. For a line bundle on a variety $U/\Qpb$ it takes the form of
an element $\sum_i [\eta_i]\otimes \omega_i \in \hdr^1(U)\otimes
\Omega^1(U)$,  and when $U$ is affine it is related with the
log-function of a section $s$ by the formula~\cite[Proposition~2.7]{Bes00}
\begin{equation}\label{dlogsec}
  d \log(s) = \sum_i \omega_i \int \eta_i + \theta
\end{equation}
up to an unknown holomorphic form $\theta$, where the $\eta_i$ are
forms of the second kind representing the cohomology classes $[\eta_i]$. For this theory we refer
to~\cite[Section~4]{Bes00}.

Following~\cite{Gross86} we extended in~\cite{Bes12} the local
$p$-adic height pairing to the case of divisors whose support is not
necessarily disjoint. At each point in the common support one uses a
tangent vector to normalize the height pairing. It is done in such a
way that choosing the tangent vectors on the global curve makes
the sum of the local height pairings independent of the choice. This
normalization is achieved by replacing the undefined term $G(x,x)$ at
a common point of support $x$ by the term $G_x(x,\cansec)$, depending
on a tangent vector $\cansec$ at $x$, which is the ``constant term''
of the function $G_x(y):=G(x,y)$ at $y=x$ with respect
to a local parameter $z$ normalized with respect to $\cansec$ in the sense
that $\partial_\cansec (z) = 1$. The constant term is defined in
general in~\cite[Definition~3.1]{Bes-Fur03}. In our situation, as $G$ has
logarithmic singularities along the diagonal, we more simply have,
by~\cite[Lemmas 3.2 and 3.3]{Bes12}
the local expansion around $x$, $G_x = \log(z)+c_0+c_1 z+\cdots $ and
$G_x(x,\cansec)=c_0$.

To analyze the resulting local height at a place above $p$ we pick a
point $x_0$ on $C$ and a tangent vector $\cansec_0$ to $C$ at
$x_0$. We can then define the function $\hoti$ on the tangent bundle $\tang$ to $C$ 
outside of the point $x_0$ by (see~\cite{Bes12})
\begin{equation*}
  \hoti(x,\cansec) = G_x(x,\cansec)-2G_{x_0}(x) + G_{x_0}(x_0,\cansec_0)\;.
\end{equation*}
Clearly, this function is related to the local height pairing as follows: If
all data are defined over a finite extension $K$ of $\Qp$, then the
value $\hoti(x,\cansec)$ lies in $K$ and the local height, evaluated
on the divisor $(x)-(x_0)$ with choices of tangent vectors $\cansec$ and
$\cansec_0$, is exactly  $\tr_v \hoti(x,\cansec)$.
 
The key observation of~\cite{Bes12} was that the function $\hoti$ is a
log-function on $\tang|_{C-x_0}$. Furthermore it is related to the log
function on  $\OO(\Delta)|_{\Delta}$, obtained as the restriction to
$\Delta$ of the canonical log-function discussed above, via the canonical
identification $\tang \cong \OO(\Delta)|_{\Delta}$.  The curvature of this
canonical log-function was computed in~\cite[Definition 5.1, Theorem
5.10]{Bes00}. Using this, we obtained in~\cite{Bes12} the following
result.
\begin{proposition}[{\cite[Proposition 3.10]{Bes12}}]\label{curveform}
  Let $\{\omega_0,\ldots, \omega_{g-1}\}$ be a basis for the holomorphic
  forms on $C$ and let $\{\bom_0,\ldots \bom_{g-1}\}\subset W $ be the
  unique dual basis with respect to the cup product. Then, the
  curvature of $\hoti$ is $ -2 \sum_{i=0}^{g-1} \bom_i
  \otimes \omega_i$.
\end{proposition}

Suppose now that $C$ is a hyperelliptic curve defined by the equation
$y^2=f(x)$ with $f$ a polynomial of degree $2g+1$. We have a basis
$\{\omega_i\}_{i=0}^{2g-1}$ of the de
Rham cohomology $\hdr^1(C/\Qpb)$ given by the forms of the second kind
$\omega_i = x^i dx/2y$. The form $\omega_i$ has order $2g-2 -2i$ at
$\infty$, and it has no pole away from $\infty$. In particular, the
forms $\omega_i$ for $i\le g-1$ are
holomorphic. We let $\omega_i^\prime= \omega_{2g-1-i}$ for $0\le i \le
g-1$, so that $\omega_i^\prime$ has a pole of order $2(g-i)$ at
$\infty$. As in the introduction, we fix $\omega=\omega_0$. This
form vanishes to order $2g-2$ at infinity and has no
other zeros or poles. It determines, by duality, a section $\cansec$ of $\tang$
with a pole of order $2g-2$ at infinity and no other zeros or
poles. On the other hand, the form $\omega_{g-1}$ has no zero or pole
at infinity and we fix $\cansec_0$ to be the dual of its value there. We pick
a parameter $z= -y/(x^{g+1}\atp)$ at infinity, where $\atp$ is the leading
coefficient of $f$.
Let $\frev(x)$ denote the polynomial $x^{2g+1}f(1/x) \in \Z[x]$.
Using~\cite[(10)]{Bes-Bal10} as well as the description at infinity of
the $\omega_i$ following (11) there, and
noting that the parameter $t$ there is $-z\atp$, we find a local
equation at $\infty$ with respect to the parameter $s=x^{-1}$:
\begin{equation*}
  z^2 \atp = \atp^{-1} s \frev(s) = s + O(s^2)\;.
\end{equation*}
Differentiating, we find that 
\begin{equation*}
  -\frac{ds}{2t} = \frac{ds}{2\atp z} = (1+O(z)) dz
\end{equation*}
and
\begin{equation*}
  \omega_{i} = - s^{g-1-i} \frac{ds}{2t} = (\atp z^2)^{g-1-i} (1+O(z)) dz\;.
\end{equation*}
In particular $\omega_{g-1}=  (1+O(z)) dz$,
and one easily finds that
$z$ is normalized with respect to $\cansec_0$. Let us also record then the
local expansion
\begin{equation}\label{xintermss}
  x= s^{-1} = \atp^{-1} z^{-2} (1+\cdots )\;.
\end{equation}
As the $\omega_i$ only have poles at $\infty$, the formula for the cup
product in terms of integrals and residues~\cite[Corollary 5.1]{Col98} gives
\begin{equation*}
 [\omega_j^\prime] \cup [\omega_i] = \Res_\infty\left(\omega_i \int \omega_j^\prime\right)\;.
\end{equation*}
Clearly, when $g-1 \ge j>i$ the form $\omega_i \int \omega_j^\prime$
is holomorphic at $\infty$, while
\begin{align*}
  \omega_i \int \omega_i^\prime &= \atp^{g-1-i} z^{2(g-1-i)} (1+O(z))
  \frac{\atp^{i-g}}{2(i-g)+1} z^{2(i-g)+1} (1+O(z)) dz\\ &= \frac{dz}{z
    \atp (2(i-g)+1)}  (1+O(z))\;.
\end{align*}
We find
\begin{equation}
  \label{eq:cup}
  [\omega_i^\prime] \cup [\omega_i] = \frac{1}{\atp(2i+1-2g)}\;,\;
  [\omega_j^\prime] \cup [\omega_i] = 0 \text{ if } g-1 \ge j>i\;.
\end{equation}

Let  $\{\bom_i\}_{i=0}^{g-1}$ be the basis for $W$ which is dual to the basis
$\{\omega_i\}_{i=0}^{g-1}$ for the holomorphic forms via the cup
product. From~\eqref{eq:cup} it follows that
\begin{equation}
  \bom_i =  \atp (2i+1-2g)\omega_i^\prime+ \text{ combination of }
  \omega_j^\prime \text{ for } j>i\; +  \text{ a holomorphic
    form.}\label{eq:bom}
\end{equation}

Recall~\cite[Definition~7.7]{Bes-deJ02} that the constant term of a
Coleman integral with respect to a parameter $z$ at a point is the
coefficient of $1$ in its expansion which includes both powers of $z$
and non-negative powers of $\log(z)$. We are going to normalize our
integrals in such a way that the constant term with respect to the
parameter $z$ at $\infty$ is $0$. We will write $\int_{\cansec_0}$ (recall
that $z$ is normalized with respect to $\cansec_0$) for this
type of integral We note that this is the integral from the tangential
base point $\cansec_0$ when all integrands have
logarithmic singularities at
$\infty$~\cite[Proposition~2.11]{Bes-Fur03}.

\begin{theorem}\label{tauthm}
  Let $\tau$ be the pullback of $\hoti$ under $\cansec$. Then we have
  \begin{equation}\label{Thm21eq}
    \tau(x) = -2 \int_{\cansec_0}^x \left(\sum_{i=0}^{g-1} \omega_i \int_{\cansec_0}^y
      \bom_i \right) - (g-1) \log(\atp)\;.
  \end{equation}
\end{theorem}
\begin{proof}
By \eqref{dlogsec} and Proposition~\ref{curveform} we have
  \begin{equation*}
    d \tau = -2 \sum_{i=0}^{g-1} \omega_i \int \bom_i+\theta
  \end{equation*}
  where $\theta$ is a holomorphic form on $U$ and the integrals are indefinite.
  There is the indeterminacy of the constant of integration and of the
  form $\theta$. Note that the constant of integration for the $\bom_i$
  simply adds to $\theta$ holomorphic forms on $C$. Thus we have
  \begin{equation*}
    \tau(x) + 2 \int_{\cansec_0}^x \left(\sum_{i=0}^{g-1} \omega_i \int_{\cansec_0}^y
      \bom_i \right) = \int \theta'
  \end{equation*}
  where $\theta'$ is holomorphic on $U$. We would like to show that
  the right hand side is a constant and then compute this constant.
  
  First we argue that the form $\theta'$
  extends to a holomorphic form on $C$. To see this, we observe that
  it follows from~\eqref{eq:bom} that $\omega_i (\int \bom_i )$ has a
  simple pole at
  infinity and no other pole. On the other hand, since $\omega_0$ has no
  poles or zeros outside infinity,  $\cansec(x)$ as a section of $\tang$
  also has no such zeros and poles outside infinity. It follows that
  the form $d\tau$
  also has a simple pole at infinity and no other pole. Indeed,
  by~\cite[Definition~4.1]{Bes00} and the ensuing remarks,  if $s$ is a
  non-vanishing section of $\tang$,
  then locally for the analytic topology, $\tau$ looks like $\loti(s)
  +\log(\cansec/s)$ with $\loti(s) $ analytic. Thus, $\theta'$ has a simple pole at
  infinity, but since it is  meromorphic, the residue theorem implies that it is
  in fact holomorphic on $C$ (one can directly compute the residue at
  infinity and easily discover that it is $0$).
  
  Next we claim that $\theta' = 0$. This is because both $\tau$ and the
  integral on the right hand side of~\eqref{Thm21eq} are symmetric with respect to the hyperelliptic
  involution $w$. For $\tau$ this follows from functoriality of the
  local height
  pairing. Namely, as the involution preserves the complementary space
  $W$, the Green function $G$, which depends only on $W$ and the branch
  of the $p$-adic log, is preserved as well. The tangent vectors used
  for normalizations are multiplied by $-1$ but since $\log(-1)=0$, the
  normalization is maintained as well by~\cite[Proposition 3.4 and
  Definition 2.1 (iii)]{Bes12}.  For the
  integral, we first note that all forms considered are
  anti-symmetric with respect to $w$. Then the integral of $\bom_i$ is
  anti-symmetric up to a constant, and this is $0$ because the constant
  term with respect to a parameter which is itself anti-symmetric is
  $0$. Thus, all terms $\omega_i \int_{\cansec_0}^y \bom_i$ are symmetric, and
  their integral is so as well.
  Consequently, $\theta'$ is a holomorphic symmetric form on $C$, hence
  $0$.
  
  It follows that~\eqref{Thm21eq}
  holds up to a constant. It suffices to show that the constant term of
  $\tau$ at infinity is $-(g-1) \log(\atp) $. With $\tau$ replaced by the
  pullback of $\hoti$ under the dual of $\omega_{g-1}$, the constant
  term is $0$ by~\cite[Proposition~3.10]{Bes12}, since
  $\omega_{g-1}$ is dual to the non-vanishing $\cansec_0$ at infinity. However,
  this pullback differs from $\tau$ by $\log(x^{g-1})= (g-1)
  \log(x)$. Using~\eqref{xintermss} we find  
  \[
  \log(x) = -2\log(z) -\log(\atp) +\;\textrm{a power series vanishing at 0}.
  \] 
  The constant term of $x$ with respect to $z$ is therefore $-\log(\atp)$ and that of $\tau$ is $-(g-1)\log(\atp)$.
\end{proof}

\section{Application to $p$-integral points}
\label{sec:integral}

The results of the previous section may be used, in a similar way to
the results in~\cite{Bes12}, to obtain a $p$-adic characterization of
$p$-integral points on hyperelliptic curves, a kind of ``quadratic
Chabauty.'' Let $f\in\Z[x]$ be a separable polynomial of
degree $2g+1 \ge 3$ such that $f$ does not reduce to a square modulo $q$ for any
prime number $q$.
Let $\UU=\Spec(\Z[x, y] / (y^2 - f(x)))$, so that $\UU(\Zip)$ is exactly the set of
$p$-integral solutions to $y^2=f(x)$. Let $X$ be 
the normalization of the closure of
the generic fiber of $\UU$, and let $J$ be its Jacobian.
\begin{theorem}\label{cutting-function}
  Let $f_i$, for $i=0,\ldots, g-1$, be defined by the formula
  \begin{equation*}
    f_i(z) = \int_\infty^z \omega_i,
  \end{equation*}
  and let $g_{ij} = f_i \cdot f_j$ for $i\le j$. Suppose that the
  Mordell-Weil rank of $J$ is exactly $g$. Then either there exists a
  linear combination of the $f_i$ that vanishes on $\UU(\Zip)$ (in other
  words, Chabauty's method works, at least for the $p$-integral points), or
  there exist
  constants $\alpha_{ij}\in \Q_p$ such that the function
  \begin{equation}\label{rhofunction}
    \rho(z) = \tau(z) - \sum_{i\le j} \alpha_{ij} g_{ij}(z)
  \end{equation}
  satisfies the following:
  \begin{enumerate}
  \item There is an effectively computable constant $s \in \Q_p$ such that we have
      $\rho(x) = s$ for every $x\in \UU(\Zip)$
    which does not intersect any of the singular points in any of the bad
    fibers of $\UU$. If $f$ is monic, then we have $s = 0$. \label{31a}
  \item There exists an effectively computable  finite set of values $T$ containing 
      $\rho(\UU(\Zip))$.
  \end{enumerate}
\end{theorem}
\begin{proof}
  We recall that the height pairing depends on the choice of a
  $\Qp$-valued idele class character $\ell=\oplus \ell_q$, a sum over
  all finite primes $q$. Fixing a choice for the $p$-adic logarithm,
  which we simply denote by $\log$, we can choose $\ell$ in such a way
  that $\ell_p= \log$. In this case, the trace map mentioned in the
  introduction is simply the identity, and the local height pairing
  $h_p$ is simply the pairing mentioned there, so that
  $h_p((x)-(\infty),(x)-(\infty))$, normalized as before, is simply
  $\tau(x)$.
  
  If Chabauty's method does not work, the $f_i$, extended linearly,
  induce linearly
  independent $\Qp$-valued functionals on the $g$-dimensional vector
  space $J(\Q)\otimes \Q$. It follows that the $g_{ij}$ are a basis for
  the space of $\Qp$-valued quadratic forms on $J(\Q)\otimes \Q$. The height pairing
  $h$, being another $\Qp$-valued quadratic form on $J(\Q)\otimes \Q$, is thus a
  linear combination of the above basis, $h= \sum \alpha_{ij}
  g_{ij}$. If $x\in X(\Q)$, then, applying the last formula to the divisor $(x)-(\infty)$, 
  Theorem~\ref{tauthm} implies that
  \begin{equation}\label{rhoequation}
    \rho(x) = -\sum_{q\ne p} h_q((x)-(\infty),(x)-(\infty))\;.
  \end{equation}
  Note that the local heights $h_q$ are computed using the localization
  of the same global tangent vectors as before.
  According to Proposition~\ref{locaway} below,  for each prime number $q\ne p$
  there is an effectively computable $p$-adic number $s_q$ such that 
  when $x$ is a $p$-integral point and either the
  reduction at $q$ is good, or, more generally, the reduction is bad but
  $x$ does not reduce to a singular point modulo $q$, we
  have \[ h_q((x)-(\infty),(x)-(\infty)) = s_q.\]
  See~\eqref{aq} for an explicit expression for $s_q$.
   If $q$ does not divide the leading coefficient $\atp$
  of $f$ then $s_q = 0$ by Proposition~\ref{locaway}.

  In general, Proposition~\ref{locaway} 
  implies that there is a proper regular model $\XX$ of $X \otimes \Q_q$ over $\Z_q$
  such that 
  if $x$ is $p$-integral, then $ h_q((x)-(\infty),(x)-(\infty))$ depends solely (and explicitly, see
  Section~\ref{explicit-intersection}) on the component of the special fiber
  $\XX_q$ that the section in $\XX(\Z_q)$ corresponding to $x$ intersects. Thus, for any $p$-integral point $x$, the right hand side
  of \eqref{rhoequation} can only take a finite number of explicitly
  computable values, completing the proof.
\end{proof}
\begin{remark}\label{comprem}
  From the proof it is clear that the size of the set $T$ is bounded by
  $1+\prod_q (m_q-1)$, where $m_q$ is the number of multiplicity one
  components of the special fiber $\XX_q$ of the proper regular model $\XX$ of
  $X$ constructed below, since these are the only components
  through which the section corresponding to $x$ can pass.
\end{remark}
We construct a normal model $\XX'$ of $X \otimes \Q_q$ over $\Spec(\Z_q)$ as follows: 
If $F(X, Z)$ is the degree $2g+2$-homogenization of $f$, then the equation
\[
Y^2 = F(X,Z)
\]
gives a smooth plane projective model of $C$ in projective space
$\PP^2_{\Q_q}(1,\,g+1,\,1)$ over
$\Q_q$ with respective weights $1,\,g+1$ and $1$ assigned to the variables
$X,\,Y,\,Z$.
Let $\XX'$ be the scheme defined by the same equation in
$\PP^2_{\Z_q}(1,\,g+1,\,1)$.
We call $\XX'$ the {\em Zariski closure}\/ of $X \otimes \Q_q$ over $\Spec(\Z_q)$.

By~\cite[Corollary~8.3.51]{liu:agac}, there exists a proper regular model $\XX$ 
of $X \otimes \Q_q$ over $\Z_q$, together with a proper birational morphism 
$\phi: \XX \to \XX'$  that is an isomorphism outside the singular locus of $\XX'$.
Such a model is called a {\em desingularization in the strong sense}\/ of $\XX'$.
If $y$ is a $\Q_q$-rational point on $X$, then, by abuse of notation, we also denote the corresponding
section $y \in \XX(\Z_q)$. 
Our assumptions on $f$ guarantee that there is a unique component $\Gamma_0$ of
the special fiber of $\XX$ which dominates the special fiber
of $\XX'$, since $\XX'$ is normal with an irreducible and reduced special fiber.
Then any $y \in X(\Q_q)$  whose reduction modulo $q$ is nonsingular has
the property that the section $y$ intersects $\Gamma_0$.

\begin{proposition}\label{locaway}
  Let $\XX$ be as above, and let $x$ be a $p$-integral point. 
  The value of $h_q((x)-(\infty),(x)-(\infty))$ depends only on the
  component $\Gamma_x$ of the fiber of $\XX$ above $q$ that $x$ passes through and is
  effectively computable from $\Gamma_x$.
  Moreover, this value  is $0$ if $\Gamma_x = \Gamma_0$ and $q$ does not divide $\atp$.
\end{proposition}
\begin{proof}
  Let us first recall how $ h_q((x)-(\infty),(x)-(\infty))$ is
  computed~\cite[Section~5]{Gross86}. 
  Let $(\;\,.\,\;)$ denote the rational-valued intersection multiplicity on $\XX$.
  According to~\cite[Theorem~III.3.6]{Lan88}, there is a vertical $\Q$-divisor $\Phi((x) -
  (\infty))$ on $\XX$  such that $D_x = x - \infty +\Phi((x) -
  (\infty))$ satisfies
  $(D_x\,.\, \Gamma) = 0$ for all vertical divisors
  $\Gamma \in \Div(\XX)$.
  The local height pairing $h_q((x) - (\infty), (x) - (\infty))$ is then given by
  \begin{equation}\label{hq-formula}
    h_q((x) - (\infty), (x) - (\infty)) = -(D_x\,.\,D_x)\log q\;.
  \end{equation}
  Up to addition of a rational multiple of the entire special fiber $\XX_q$, which
  is irrelevant for intersections with vertical divisors, $\Phi((x) - (\infty))$ only
  depends on which components the sections $x$ and $\infty$ pass through, see also 
  Section~\ref{explicit-intersection}.
  Note that the intersection $(D_x\,.\,D_x)$ is equal to
  \[
      (x - \infty)^2 + \Phi((x) - (\infty))^2.
  \]
  One computes the first intersection
  using the following rule for self-intersection of
  horizontal components: Let $x\in \XX(\Z_q)$,
  and let $\cansec$ be the chosen tangent vector at $x$. We have the following:
  \begin{itemize}
  \item If $\cansec$ is a
    generator to the tangent bundle at $x$, then $(x \,.\, x) = 0$. 
  \item 
    More generally, if $\alpha \cansec$ is a generator, with $\alpha \in \Q_q$, then
    $(x \,.\, x) = -v_q(\alpha)$,\end{itemize} where $v_q$ is the $q$-adic valuation
  (see 3 of Definition~2.3 in~\cite{Bes12}). 
  We note that   the $p$-integrality of $x$ implies that $(x \,.\, \infty) = 0$.
  
  In our case the tangent vector at $x$ is determined as the dual
  to the value of the differential form $\omega$ at that point. 
  Lemma~\ref{withomdiv}~(ii) below implies that the intersection multiplicity $(x \,.\, x)$ depends only
  on $\Gamma_x$.
  This proves the first statement of the proposition.
  
  Now suppose that $x$ intersects $\Gamma_0$, then $(x \,.\, x)  = 0$
  follows from Lemma~\ref{withomdiv}~(iii).
  If we assume, in addition, that $q$ does not divide $\atp$, then 
  $\infty$ reduces to a nonsingular point and hence the
  corresponding section intersects $\Gamma_0$.
  Since  $x$ and $\infty$ pass
  through the same component, it is clear that we can take $\Phi((x) - (\infty)) = 0$.
  Finally, Lemma~\ref{withomdiv}~(iv) implies that $(\infty \,.\,
  \infty) = 0$.
\end{proof}
\begin{lemma}\label{withomdiv}
  Let $\XX$ and $\omega$ be as above and write the divisor $\div(\omega)
  \in \Div(\XX)$ as $\div(\omega) = H + V$, where $H$ is horizontal and
  $V$ is vertical. Let $x\in X(\Q_q)$.
  \begin{itemize}
  \item[(i)] We have $H = (2g - 2) \cdot \infty$.
  \item[(ii)] For the intersection multiplicity, as normalized by our chosen tangent
    vectors, we have $(x \,.\, x) = - (x \,.\, \div(\omega))$. In particular,
    $(x \,.\, x) = - (x \,.\, V)$ if $x\in \UU(\Z_q)$. 
  \item[(iii)] The component $\Gamma_0$ is not contained in $V$.
  \item[(iv)] We have $(\infty \,.\, \infty) = - (\infty \,.\, W)$,
       where $W$ is the vertical part of $\div(\omega_{g-1})$ and
      the self-intersection is taken with respect to $\cansec_0$.
      If $q$ does not divide $\atp$, then this is equal to~0.
  \end{itemize}
\end{lemma}
\begin{proof}
  The proof of (i) is obvious. 
  We now turn to the proof of (ii).
  Using cotangent vectors, the formula for the self-intersection $(x \,.\, x)$ is as follows:
  if $\alpha \omega(x)$ is a generator for the cotangent bundle at $x$, with $\alpha \in \Q_p$, then 
  $(x \,.\, x) = v_q(\alpha)$. But this is
  exactly 
  $ -(x \,.\, \div(\omega) )$. Indeed, for any line bundle $\LL$ on
  $\XX$ and any  meromorphic section $s$ of $\LL$ that has no zeros or
  poles at $x$ on the generic fiber, the pullback $x^*\LL$ is a free
  $\Z_q$-module of rank~1 and $s(x)$ is a non-zero element of
  $x^*\LL\otimes \Q_q$. When $s(x)\in x^*\LL $ and $\alpha \cdot s(x)$ is a generator for $x^*\LL$ we have
  \[
  v_q(\alpha) = -\log_q \#(x^*\LL / s(x) \Z_q),
  \]
  and the right hand side is well-known to be equal to $ -(\div(s) \,.\,
  x)$. The behavior of intersection with a principal divisor implies that
  this continues to hold without the assumption $s(x)\in x^*\LL $.
   When $x\in \UU(\Zip)$ we know that $(x \,.\, \infty) = 0$, so 
  using (i) we get $(x \,.\, x )= - (x \,.\, V)$.
  
  For (iii), note that   on the integral affine subscheme    
  $\UU\otimes \Z_q$ of $\XX'$, both the relative cotangent
  bundle and the relative dualizing sheaf over $\Spec(\Z_q)$ are generated by 
  $\omega|_{\UU\otimes \Z_q}$, see~\cite[\S6.4]{liu:agac}. 
  Hence $\cansec$ generates the tangent bundle on $\UU$. But since
  $\XX'$ and $\XX$ are isomorphic outside the singular locus of $\XX'$ by
  assumption, (ii) implies
  \[
  0 = (x\, . \,x) = -(V\, .\, x)
  \]
  for any $\Q_q$-rational point $x$ such that $x$ does not reduce to
  $\infty$ or a singular point modulo $q$. 
  So $\div(\omega)$ cannot contain $\Gamma_0$, since all such points reduce
  to $\Gamma_0$.
  
  The proof of the first part of (iv) is analogous to the proof of (ii).
  If $q$ does not divide $\atp$, then $\infty$ is nonsingular modulo $q$.
  The form  $\omega_{g-1}$ is just $-du / 2v$
  on the affine patch $U_\infty$ of $X$ given
  by $v^2 = \frev(u)$, where $u = 1 / x$ and $ v = y / x^{g+1}$.
  So (iv) follows in a manner similar to (iii)
  because $du / 2v$ generates $\Omega^1_{U_\infty / \Spec(\Z_q)}$.
\end{proof}
\begin{remark}\label{g1minimal}
  If $g=1$ and the given equation of $C$ is a minimal Weierstrass equation,
  then we can take $\XX$ to be the minimal regular model of $C$.
  In this case it is easy to show that $\div(\omega) = 0$.
\end{remark}
\begin{remark}
  The computation of the function $\rho$ requires the following
  ingredients:
  \begin{itemize}
  \item Coleman integration, including iterated Coleman integrals~\cite{balakrishnan:iterated}, for
    the computation of $\tau$;
  \item computation of $p$-adic height pairings~\cite{Bes-Bal10,BaMuSt12};
  \item a basis for the Mordell-Weil group of $J$ tensored with
    $\Q$. Typically one computes this by searching for points of small height in
    $J(\Q)$ (or differences of rational points on $X$) until we have found $g$
    independent points.
    Note that we first need to verify that the rank of $J(\Q)$ is indeed $g$, for
    instance using~2-descent on $J$, cf. \cite{stoll:2-descent}.
  \end{itemize}
  Clearly, computing the height pairings between all pairs of elements
  of the above basis, together with the computation of the integrals of
  the $\omega_i$ on elements of this basis, suffices for the
  determination of the constants $\alpha_{ij}$. Integrals of holomorphic
  forms give the $f_i$'s, hence the $g_{ij}$'s, and iterated Coleman
  integrals give $\tau$. Note that one can get bounds on the number of
  integral points without computing iterated integrals, see
  Remark~\ref{boundingredients}.
\end{remark}

\begin{remark}\label{practical}
  Here we describe how to use these ideas to give an algorithm to find integral points on a genus $g$ hyperelliptic curve $X$ with Mordell-Weil rank $g$:
  \begin{enumerate}
  \item Let $D_1, \ldots, D_g \in \Div^0(X)$ be representatives of the elements of a basis for the Mordell-Weil group of the Jacobian
    tensored with $\Q$.  We compute the global $p$-adic height pairings
    $h(D_1,D_1),\, h(D_1,D_2),\ldots, h(D_g, D_g)$ and the $\frac{g(g+1)}{2}
    \times \frac{g(g+1)}{2}$ matrix  with entries  $\frac{1}{2}(f_k(D_i)f_l(D_j) + f_l(D_i)f_k(D_j))$
    for $1 \le i \le j \le g$ and $0 \leq k \leq l \leq g-1$. Solving a linear system (see Example~\ref{ex:g2}) gives the vector of $\alpha_{ij}$ values in \eqref{rhofunction}.
  \item Compute the finite set $T$ of possible values that $\rho$ can take on $p$-integral
    points. 
  \item Compute the value of $\tau(P)$ for some point $P$ on the curve. Use \eqref{splittrick} and the power series expansions of the double and single Coleman integrals  to give a power series describing $\tau(t)$ in each residue disk.
  \item The integral points are solutions $t$ to $\tau(t) - \sum
    \alpha_{ij}f_i(t)f_j(t) = a$
    across the various residue disks, where $a$ runs through the elements of
    $T$.
  \item Use the Mordell-Weil sieve to identify solutions that
    correspond to integral points and recover these points.
  \end{enumerate}
  We show how to carry out Steps 1--4 in practice in
  Section~\ref{examples}. Step 5 will be discussed in future work.
\end{remark}

\begin{remark}\label{higherfields}
  It would be interesting to try to use height pairings to find integral
  points for number fields other than $\Q$. Note that the height
  pairing provides a $\Qp$-valued equation. This has to be combined
  with $\Qp$ equations obtained using the extension of the explicit Chabauty
  method to number fields~\cite{Sik13} to get at least as many
  equations as the degree of the field, which is the $\Qp$-dimension
  of the manifold obtained by tensoring with $\Qp$. In certain situations, one may
  obtain more than one equation from the height pairing by using more
  than one idele class character. Let us describe one simple situation
  where this can be achieved (this will be studied in greater generality in forthcoming work).
  
  Suppose $L$ is an imaginary quadratic extension with class number
  $1$ where the prime $p$ remains inert. Since the group of units is
  torsion, one can lift any $\ell_p: L_p^\times \to \Qp$ to an idele class character: A uniformizer at the
  (principal) prime ideal $(\alpha)$ will be sent to
  $-\ell_p(\alpha)$. Composing the $p$-adic log with a basis of the
  $\Qp$-linear functionals $L_p \to \Q_p$, we get two different $\ell_p$,
  hence two idele class characters, from which we get two $p$-adic
  equations.
\end{remark}
\section{Computation of $\tau$}\label{sec:examples}

To obtain numerical examples, we need to compute the function $\rho$
from \eqref{rhofunction}. This function involves height function
computations for the determination of the constants $\alpha_{ij} $
appearing there, for which we have the algorithm of
\cite{Bes-Bal10,BaMuSt12}, and the computation of Coleman
integrals of holomorphic forms, which can be performed using the algorithms of~\cite{BBK09}. 

The remaining component is the computation of the function $\tau$, which
is an iterated Coleman integral \cite{balakrishnan:iterated}. In this section, we discuss how to leverage information about local $p$-adic height pairings to compute iterated integrals from tangential points as an alternative to direct computation of these integrals.

First, when $\tau$ is to be evaluated at a Weierstrass point, we extend the result of  \cite[Proposition~6.1]{Bes12} to hyperelliptic curves:
\begin{lemma}
  Let $P=(A,0)$ be a Weierstrass point. Then we have
  \begin{equation*}
    \tau(P) = \frac{1}{2} (\log(f'(A)) +\log(\atp)).
  \end{equation*}
\end{lemma}
\begin{proof}
By the properties of the local height pairing, the value of $\tau$ at $P$
is $1/2$ of the normalized value of $x-A$ on the divisor $(P) -
(\infty)$. As shown in \cite[Proposition~6.1]{Bes12}, the normalized value at $P$ is
$\log(f'(A))$, while by~\eqref{infval} the normalized value at $\infty$
is $\atp^{-1}$.
\end{proof}

In general, one way to avoid directly computing iterated integrals from tangential points is to compute the iterated integral with one specific endpoint\footnote{In this section we will be using the coordinate functions extensively. We therefore denote, unlike in previous sections, points on the curve by $P$, $Q$, etc.}
$P$ (chosen in a certain favorable way) in some indirect
way. The integral for a general point can then be computed from this
(see \eqref{splittrick} below). For example, in the case of elliptic curves,
in~\cite{BKK11} the point $P$ was taken to be a global two-torsion
point or a tangential point at infinity, while in~\cite{Bes12} a (non-global) two- or three-torsion point was used.

In the hyperelliptic curve case, we can also integrate from finite Weierstrass points, but this requires working over totally ramified extensions of $\Q_p$, which is, in practice, quite slow.  As an alternative approach, we describe a technique for computing
$\tau(P)$ for a general point $P$ that does not use the description of
$\tau$ in terms of iterated integrals in Theorem~\ref{tauthm} but
instead works directly with the description of $\tau$ as a local
height pairing. Given $\tau(P)$, the value of $\tau$ at any other
point $P'$ may be computed using the following formula.
\begin{equation}\label{splittrick}
\begin{split}\tau(P') &= -(g-1)\log(\atp) -2 \int_{t_0}^{P'} \left(\sum_{i=0}^{g-1} \omega_i  \bom_i \right)\\
&= -(g-1)\log(\atp) -2\sum_{i=0}^{g-1} \left( \int_{t_0}^{P}  \omega_i  \bom_i +
  \int_{P}^{P'} \omega_i  \bom_i  + \int_{P}^{P'} \omega_i
  \int_{t_0}^{P} \bom_i \right)\\
&= \tau(P) -2\sum_{i=0}^{g-1} \left( \int_{P}^{P'} \omega_i  \bom_i  + \int_{P}^{P'} \omega_i  \int_{t_0}^{P} \bom_i \right)\;\\
      \end{split}      
    \end{equation}
where we are using a decomposition of iterated integrals which may be
found, for example, in~\cite[p.~288]{BKK11}.

We note that for applications, the description of $\tau$ in terms of
local power series expansions is vital. Thus, while we are able to
compute $\tau(P')$ for any given $P'$, we will be using
\eqref{splittrick} after computing $\tau(P)$ for one, or possibly a
finite number, of values of $P$, as this last formula does in fact
give such a power series expansion.

The strategy for computing $\tau(P)$ for a particular point $P$ in a residue disk is to interpret it as a
normalized local height
and to directly compute this using an extension of the techniques
of~\cite{Bes-Bal10}.

To this end, it is first of all somewhat helpful to use cotangent
vectors rather than tangent vectors for the normalization. Let $\omega$ denote a choice, for
each point on the curve $C$, of a cotangent vector at that point. This
could obviously come from a differential form but could also be just
an arbitrary assignment. We can write $\hloc(D_1,D_2)_\omega$ for the
local height pairing of $D_1$ and $D_2$, computed with respect to the
tangent vectors dual to the cotangent vectors specified in
$\omega$. This notation has the advantage that if $f$ is an
isomorphism of curves, one
has, by an easy argument
\begin{equation*}
  \hloc(f^\ast D_1, f^\ast D_2)_{f^\ast \omega} = \hloc(D_1,D_2)_\omega
\end{equation*}
where $f^\ast \omega$ means the pointwise pullback of the cotangent
vectors with respect to the differential of $f$ at a point (this
coincides with the usual pullback for holomorphic forms). Another important
observation is that \[\hloc(D_1,D_2)_{-\omega} =
\hloc(D_1,D_2)_\omega ,\] because of the property~\cite[Definition
2.3]{Bes12} of the local height pairing. For the hyperelliptic case,
the assignment $\omega$ corresponds to one form away from infinity, and to
the value of another form there, but both are anti-symmetric with
respect to the hyperelliptic involution $w$, so with that choice of
$\omega$ we have
\begin{equation*}
  \hloc(w^\ast D_1,w^\ast D_2)_{\omega} = \hloc(D_1,D_2)_\omega\;.
\end{equation*}
Thus, with this choice of tangent vectors, the splitting into
symmetric and anti-symmetric components of~\cite{Bes-Bal10} extends
without any difficulty. In particular, we can use formula (19)
there. Using the shorthand $\hloco{D}$ for $\hloco{D,D}$, and taking
$D=(P)-(\infty)$, we have

\begin{align*}
  \tau(P) &= \hloco{(P)-(\infty)}\\
  &= \frac{1}{4}\hloco{(P)-(\infty) + w((P)-(\infty) )}
  + \frac{1}{4}\hloco{(P)-(\infty) - w((P)-(\infty) )}
  \\ 
 &=  \frac{1}{4}\hloco{(P) + (w(P))-2(\infty) }
 + \frac{1}{4}\hloco{P - (w(P))}\;.
\end{align*}

Suppose now that $P=(a,b)$ is a point with non-Weierstrass reduction.
We compute $\hloco{(P)+ (w(P))-2(\infty) }$. Since
$(P)+ (w(P))-2(\infty) $ is the divisor of the function
$x-a$, this local term is $\log$ on the normalized
value of $x-a$ on $(P)+ (w(P))-2(\infty)$. To compute
this, we note that a normalized parameter with respect to $\omega$ at
$P$ is
\begin{equation}\label{eq:normpar}
  z=\frac{x-a}{2b},
\end{equation}
and so the normalized value $(x-a)[P]$ is
just $2b$. Similarly, the normalized value at $w(P)$ is $-2b$. On the other hand, using the  normalized parameter $-y/(\atp x^{g+1})$ with
respect to $t_0$ at infinity, the
normalized value of $x-a$ at infinity is
\begin{equation}\label{infval}
  \lim_{x\to \infty} (x-a)\frac{y^2}{\atp^2 x^{2g+2}}= \lim_{x\to \infty}
  (x-a)\frac{f(x)}{\atp^2 x^{2g+2}}=\atp^{-1},
\end{equation}
recalling that $f$ has leading coefficient $\atp$.
Collecting all the data and using~\cite[Definition~2.3]{Bes12} we find
\[\hloco{(P) + (w(P))-2(\infty) }= \log(-4b^2\atp^2)= \log(4b^2) + 2\log(\atp) \;.\]

Next we compute the anti-symmetric part, $\hloco{(P) -
  (w(P))}$. By~\cite[Proposition~3.4]{Bes12}, the local height pairing
for two divisors (which are not necessarily disjoint) is given by
\begin{equation*}
  \hloco{D_1,D_2}= G_{D_1}[D_2]
\end{equation*}
where, by~\cite[Theorem~7.3]{Bes00}, if the degree of $D_1$ is $0$, the Green
function $G_{D_1}$ coincides with with the Coleman integral $\int
\omega_{D_1}$. Here $\omega_{D_1} $ is, as in~\cite{Bes-Bal10}, a
certain form of the third kind whose residue divisor is $D_1$, and the
value of this integral on $D_2$ is normalized by taking constant terms
with respect to local parameters which are normalized with respect to
the chosen cotangent vectors.

In our case we have $D_1=D_2=D = (P)-(\infty)$. As
in~\cite[Algorithm~5.8]{Bes-Bal10}, we can decompose $\omega_D$ as
$\nu-\eta$, where $\nu$ is an arbitrary form of the third kind with
residue divisor $D$, while $\eta$ is holomorphic and is computed from
$\nu$ and from the decomposition~\eqref{eq:decompi} as
in~\cite[Subsection~5.6]{Bes-Bal10}. We consequently have
\[\hloco{P - (w(P))}  = \int_{w(P)}^{P} \nu - \int_{w(P)}^{P} \eta\;.\]

The integral of $\eta$ can be computed using the techniques in \cite{BBK09}.  
The integral  $\int_{w(P)}^{P} \nu$ has to be normalized since $\nu$
has poles at the endpoints $P$ and $w(P)$. To properly account for this
normalization, we use an auxiliary point $Q$ in the same residue disk
as $P$. 

More concretely, by~\cite[Proposition~5.13]{Bes-Bal10} we may take $\nu =
\frac{b dx}{y(x-a)}$, which is anti-symmetric. Using this and
breaking the path of integration into several pieces, we obtain
\begin{align*}
  \int_{w(P)}^{P} \nu = \int_{w(P)}^{Q} \nu + \int_Q^{P} \nu &=  \int_{w(Q)}^{P} \nu + \int_Q^{P} \nu ,\\
  &= \int_Q^{P} \nu + \int_{w(Q)}^Q \nu + \int_Q^{P} \nu \\
  &= -2\int_{P}^Q \nu + \int_{w(Q)}^Q \nu\;.
\end{align*}
The first integral is a
tiny integral, but again, since it has a pole at $P$, this has to be normalized with respect to the chosen cotangent vector, which we do below. The second integral was considered
in~\cite[Algorithm~4.8]{Bes-Bal10}; here we make a small modification, also discussed below, to handle a subsequent part of the computation (that of a tiny integral with a pole within the disk of integration).

\subsection{Computing  $\int_{P}^Q \nu$} 
Let $z$ be a parameter
at $P$, e.g., \eqref{eq:normpar}, which is normalized with respect to $\omega$. Writing
$\nu$ in terms of $z$ we get $\nu= (z^{-1} +a_0 +a_1z +\cdots)
dz$, since $\nu$ has a simple pole with residue $1$ at $P$. The normalized
integral $\int_{P}^Q \nu$ is the normalized integral of $\nu$
evaluated at $Q$. The normalization means that the constant term with
respect to $z$ of the integral is $0$, i.e., that it is of the form
$\log(z)+ a_0 z + a_1 z^2/2+ \cdots$. If we use the parameter $t=(x-a)
= 2b z$ instead we can rewrite this as $\log(t) - \log(2b)$ plus some
power series in $t$, where this power series is nothing but the term-by-term integral of $\nu- t^{-1} dt$. Thus, in terms of the
parameter $t$, the formula for the integral is
\begin{equation*}
  \int_{P}^Q \nu = \log(t(Q))- \log(2b)+ \int_0^{t(Q)} (\nu-
  t^{-1} dt).
\end{equation*}

\subsection{Computing  $\int_{w(Q)}^Q \nu$}By~\cite[(14)]{Bes-Bal10}, we know 
\begin{equation}\label{alpha}\int_{w(Q)}^Q \nu = \frac{1}{1-p}\left(\Psi(\alpha) \cup \Psi(\beta) + \sum_{A\in\mathcal{S}} \Res_A
    \left(\alpha\int\beta\right) -2  \int_{Q}^{\phi(Q)}\nu\right), \end{equation}
    where $\alpha = \phi^*\nu - p \nu$ (a form constructed via a $p$-power lift of Frobenius $\phi$), $\mathcal{S}$ is the set of closed points, $\beta$ has residue divisor $Q - w(Q)$, and $\Psi$ is a logarithm map. 

Each of the quantities in \eqref{alpha} can be computed using the techniques in \cite{Bes-Bal10}, except for $\Res_{A\in\mathcal{S}} \left(\alpha\int\beta\right)$, since $\beta$ has poles in certain residue disks which, by construction, are disks which contain points in $\mathcal{S}$ -- i.e., disks where the integration will take place. Here we give an elementary lemma which allows us to extend the techniques of \cite{Bes-Bal10} to handle this case.
\begin{lemma}Suppose $\beta$ has residue divisor $Q - wQ$, where $Q$ is assumed to be non-Weierstrass. The integral of $\beta$ computed between points $P, P'$ distinct from $Q$ but contained in the residue disk of $Q$, written as 
\[\int_P^{P'} \beta = \int_P^{P'} \frac{f(x(Q)) - f(x)}{y(x-x(Q))(y(Q) + y)}dx+ \log\left(\frac{x(P')-x(Q)}{x(P)-x(Q)}\right)\] converges. \end{lemma}

\begin{proof}Since $P, P'$ are in the same residue disk, we compute a locally analytic parameterization $(x(t),y(t))$ from $P$ to $P'$ and use this to rewrite:
\begin{align*}\int_P^{P'} \beta &= \int_P^{P'} \frac{y(Q)  dx}{y(x-x(Q))}\\
&= \int_P^{P'} \frac{y(Q) dx}{x - x(Q)}\left(\frac{1}{y} - \frac{1}{y(Q)}\right) +\int_P^{P'} \frac{dx}{x - x(Q)}\\
&= \int_0^1 \frac{f(x(Q)) - f(x(t))}{y(t)(x(t)-x(Q))(y(Q) + y(t))}dx(t)+ \log\left(\frac{x(P')-x(Q)}{x(P)-x(Q)}\right). \end{align*}  It remains to check the convergence of the integrand.  Since $Q$ is non-Weierstrass, both $y(t)$ and $y(Q) + y(t)$ are units in $\Z_p[[t]]$. 

The claim is then that $v_p\left(\frac{f(x(Q)) - f(x(t))}{x(t)-x(Q)}\right) \geq 0$. Indeed, since $x(Q) - x(t)$ divides $f(x(Q)) - f(x(t))$, we see that $v_p\left(\frac{f(x(Q)) - f(x(t))}{x(t)-x(Q)}\right)$ is simply the $p$-adic valuation of the linear coefficient of $f(x)$, which is assumed to be integral.
\end{proof}

\section{Computing all possible values of $\rho$ on $\UU(\Zip)$}\label{explicit-intersection}
Let $q$ be a prime and let $\phi:\XX \to \XX'$ denote a desingularization in the strong sense of
the Zariski closure $\XX'$ of $C:=X\otimes \Q_q$ over $\Spec(\Z_q)$.

For $D\in\Div(C)$ we also write, by abuse of notation, $D$ for the Zariski
closure (with multiplicities) of $D$ in $\Div(\XX)$.
If $D\in\Div^0(C)$, 
then there exists a vertical $\Q$-divisor $\Phi_q(D)\in\Div(\XX)\otimes\Q$ such that 
${D}+\Phi_q(D)$ has trivial intersection multiplicities with all vertical
divisors on $\XX$, see \cite[Theorem~III.3.6]{Lan88}.
Since $\Phi_q(D)$ is vertical itself, we have
\[(\Phi_q(D)\,.\,{D})=-\Phi_q(D)^2\ge0.\]
It follows from the proof of Theorem~\ref{cutting-function},
from~\eqref{hq-formula} and from Lemma~\ref{withomdiv} that a point $x\in
\UU(\Zip)$ satisfies
\[
\rho(x) = -\sum_{q\ne p} \left(\Phi_q((x) - (\infty))^2 + (x
\,.\,V_q) + (\infty \,.\, W_q)\right)\log q\,,
\]
where $V_q$ is the vertical part of $\div(\omega)$ 
and $W_q $ is the vertical part of $\div(\omega_{g-1})$ over $q$.
Note that only bad primes can contribute toward the sum on the right hand
side.
By Lemma~\ref{withomdiv}, the $p$-adic constant $s_q$ from the proof of Theorem~\ref{cutting-function} is
therefore equal to 
\begin{equation}\label{aq}
    s_q = \left(\Phi_q((x) - (\infty))^2 + (\infty \,.\, W_q)\right)\log(q)\,.
\end{equation}
\subsection{Computing local contributions}\label{comp_vert}
In this subsection we discuss how the quantity
\begin{equation} \label{vert_ints}
    \Phi_q((x) - (\infty))^2 + (x \,.\,V_q)+ (\infty \,.\, W_q)
\end{equation}
can be computed for a given prime $q$ and $x \in \UU(\Z_q)$.
Let $M_q=(m_{ij})_{i,j}$ denote the intersection matrix of
the special fiber $\XX_q=\sum a_i\Gamma_i$.
Its entries are given by $m_{ij}=(a_i\Gamma_i\,.\,a_j\Gamma_j)$.
For simplicity, we drop the subscript $q$ in the following, as we will work over a
fixed prime $q$.

The matrix $M$ has rank $n-1$, where $n$ is the number of irreducible components
of $\XX_q$, and its kernel is spanned by the vector
$(1,\ldots,1)^T$, see \cite[\S~III.3]{Lan88}.
Let $M^+$ denote the Moore-Penrose pseudoinverse of $M$ introduced
in~\cite{penrose:inverse} and let $u(x)$ denote the column
vector whose $i$th entry is $(x-\infty\,.\,a_i\Gamma_i)$.
Then we have 
\begin{equation}\label{phi_formula}
\Phi((x)-(\infty))^2=u(x)^TM^+u(x)\,.
\end{equation}

Now we discuss the computation of $(x\,.\,V)$. 
Recall that $\div(\omega) = H + V$, where $H = (2g-2)\cdot \infty \in \Div(\XX)$ is horizontal and $V$ is vertical.
We call a $\Q$-divisor $\KK$ on $\XX$ a {\em canonical $\Q$-divisor on $\XX$}\/ if
$\OO(\KK)\cong \Bom_{\XX/\Spec(\Z_q)}$, where $\Bom_{\XX/\Spec(\Z_q)}$ is the
relative dualizing sheaf of $\XX$ over $\Spec(\Z_q)$.

Since  $\div(\omega)|_{X}=(2g-2)\cdot (\infty)$ is a canonical divisor on $X$, it
follows from \cite[Proposition~2.5]{curilla-kuehn:fermat} that we can extend $H$ to a canonical $\Q$-divisor
$\KK = H + V'$ on $\XX$ if $V'$ is a vertical $\Q$-divisor such that $\KK$ satisfies
the adjunction formula
\begin{equation}\label{adj-form}
    (\KK\,.\,\Gamma) = -\Gamma^2 + 2p_a(\Gamma) - 2
\end{equation}
for all components $\Gamma$ of the special fiber of $\XX$, where $p_a$ denotes
the arithmetic genus.
Such an extension always exists (see
\cite[\S2]{kuehn-mueller:lower_bounds} for a more general
statement), but it is not unique, because the validity of~\eqref{adj-form} is
unchanged if we add a multiple of the entire special fiber $\XX_q$ to $V'$.
Under the additional condition that $V'$ does not contain the component
$\Gamma_0$, the $\Q$-divisor $V'$ is uniquely determined.
Recall that $\Gamma_0$ is the unique component of
the special fiber of $\XX$ which dominates the special fiber
of $\XX'$.

It is not true in general that $\div(\omega)$ is a canonical divisor on $\XX$,
since $\Bom_{\XX/\Spec(\Z_q)}$ may differ from the relative cotangent bundle
$\Omega^1_{\XX/\Spec(\Z_q)}$. 
But restricting to the smooth locus $\XX^{\mathrm{sm}}$, we have
\[
    \Bom_{\XX/\Spec(\Z_q)}|_{\XX^{\mathrm{sm}}}\cong \Omega^1_{\XX/\Spec(\Z_q)}|_{\XX^{\mathrm{sm}}}
\]
by \cite[Corollary~6.4.13]{liu:agac}.
Hence $V-V'$ is supported in
the vertical divisors of multiplicity at least~2, which immediately implies
\[
   (x\,.\,V) = (x\,.\,V').
\]
Therefore it suffices to compute $V'$.
The adjunction formula~\eqref{adj-form} and the constraint $\Gamma_0 \notin
\supp(V')$ reduce the computation of $V'$ to 
a system of linear equations which has a unique solution by the above discussion.

The computation of $(\infty\,.\, \div(\omega_{g-1}))$ is analogous: We find a
vertical $\Q$-divisor $W'$ such that adding $W'$ to the horizontal part of
$\div(\omega_{g-1})$ gives a canonical $\Q$-divisor which does not contain
$\Gamma_0$ in its support.
Then we know that 
\[
     (\infty\,.\, \div(\omega_{g-1}))= (\infty\,.\, W) = (\infty \,.\, W').
\]

\subsection{Local contributions for $\rho$ in genus~2}\label{genus2}
Keeping the notation of the previous subsection, we now assume, in addition,
that $q$ does not divide $\atp$.
In this situation one can compute all possible values of
\[
\Phi((x) - (\infty))^2 + (x \,.\,V)
\]
if the genus of $X$ is~2 and $\XX_q$ has semistable reduction purely from the
reduction type.
The possible special fibers of minimal regular models of curves of genus~2 have been
classified by Namikawa-Ueno \cite{namikawa-ueno:classification}.
In particular, it is known that if $\XX$ is semistable, then, in the notation of~\cite{namikawa-ueno:classification}, the reduction
types of $\XX$ are either $[I_{n_1}-I_{n_2}-m]$ or $[I_{n_1-n_2-n_3}]$, where $m,n_1,n_2,n_3\ge0$
are integers.

In Table~\ref{tab:g2-semistable} we list the stable model corresponding to each reduction
type and the relation to the discriminant $\Delta$ of a minimal Weierstrass model of $X$.
\begin{table}[h!]
  \begin{center}
    \begin{tabular}{| c| l | c|}
    \hline
reduction type & stable reduction & $v_q(\Delta)$ \\
\hline
$[I_{0-0-0}]$& smooth curve of genus~2 & 0\\
\hline
$[I_{n_1-0-0}]$& genus~1 curve with a unique node & $n_1$\\
\hline
$[I_{n_1-n_2-0}]$& genus~0 curve with exactly 2 nodes & $n_1+n_2$\\
\hline
$[I_{n_1-n_2-n_3}]$& union of~2 genus~0 curves, & $n_1+n_2+n_3$\\
&intersecting in~3 points& \\
\hline
$[I_0-I_0-m]$& union of two smooth genus 1 curves, & $12m$\\
&intersecting in~1 point &\\
\hline
$[I_n-I_0-m]$& union of a smooth genus~1 curve and a & $12m+n_1$\\
&genus~0 curve with a unique node,&\\
&intersecting in~1 point& \\
\hline
$[I_{n_1}-I_{n_2}-m]$& union of~2 genus~0 curves with a unique &$12m+n_1+n_2$\\
&node, intersecting in~1 point& \\
    \hline
    \end{tabular}
  \end{center}
  \caption{Semistable reduction types in genus~2}\label{tab:g2-semistable}
\end{table}

Because of our assumptions on $f$, the reduction type $[I_{n_1-n_2-n_3}]$
cannot occur, since the reduction of $f$ modulo $q$ would have to be a square.

We now list all possible values
$\Phi((x)-(\infty))^2$, where $x \in X(\Q_q)$.
We also list all possible values $(x\,.\,V)$ and the possible sums
$\Phi((x)-(\infty))^2 + (x\,.\,V)$.
In the present case of semistable reduction, we have $V=V'$ in the notation of
the previous subsection.

First we consider reduction type $[I_{n_1-n_2-0}]$, where $n_1, n_2 \ge 0$.
It is easy to see that we have $V = V' = 0$. 
The set of all possible values for $\Phi((x) - (\infty))^2$ is 
\[
 \left\{-\frac{i(n_1-i)}{n_1}\right\}\cup\left\{-\frac{j(n_2-j)}{n_2}\right\},\quad 0\le i \le \lfloor
n_1/2\rfloor,\;0\le j \le \lfloor n_2/2\rfloor,
\]
see \cite[\S6]{mueller-stoll}.
The possible values for reduction type $[I_{n_1}-I_{n_2}-m]$, where
$n_1, n_2, m \ge0$, depend on where $\infty$ intersects the special fiber.
This information can be obtained easily from the equation of $X$.
If $\infty$ intersects a component corresponding to $I_{n_1}$, then it is easy
to see that
$(x\,.\,V)\in\{0,-1,\ldots,-m\}$.
Moreover, we have
\begin{equation}\label{phis}
   \Phi((x) - (\infty))^2 \in \left\{ -i,\;  -\frac{j \cdot
    (n_1 - j)}{n_1} ,\; -m - \frac{k \cdot (n_2 - k)}{n_2} \right\},
\end{equation}
and 
\begin{equation}\label{sums}
 \Phi((x)-(\infty))^2 + (x\,.\,V) \in \left\{ -2i,\; -\frac{j \cdot
    (n_1 - j)}{n_1},\; -2m - \frac{k \cdot (n_2 - k)}{n_2}     \right\},
\end{equation}
where $i\in\{0,\ldots,m\},\; j\in \{0,\ldots,\lfloor n_1/2\rfloor\}$ and $ k
\in\{0,\ldots,\lfloor n_2/2\rfloor\}$.
If $\infty$ intersects a component corresponding to $I_{n_2}$, then
we get all possible values upon swapping both $n_1$ and $n_2$ and $j$ and $k$ in~\eqref{phis}
and~\eqref{sums}.
\begin{remark}
    If we have two singular points in the reduction of $X$ modulo $q$ which are 
    conjugate over $\mathbb{F}_{q^2}$, then we know
    that $x$ and $\infty$ intersect $\Gamma_0$ and hence $\Phi((x)-(\infty))^2 +
    (x\,.\,V) = 0$.
\end{remark}
\begin{remark}
Note that the denominators are contained in $\{1, n_1,
n_1n_2\}$.
\end{remark}
\begin{remark}\label{nonmonic}
If $q$ divides $\atp$, then both $V$ and $\Phi((x) - (\infty))^2$ depend on where $\infty$ intersects the special
fiber; moreover, $W$ depends on the horizontal part of $\div(\omega_{g-1})$.
Hence none of the summands in~\eqref{vert_ints} can be read off purely from
the reduction type and the component that $x$ intersects.
Nevertheless, for a given curve $X$ the computation of all possible values of $\rho$ on $p$-integral
points is not more difficult using the techniques presented in
Subsection~\ref{comp_vert}.
\end{remark}

\section{Bounds on the number of $p$-integral points}
\label{sec:bounds}

In this section we sketch how one may use the techniques of this paper
to obtain a bound on the number of $p$-integral points on a hyperelliptic
curve. We hope that these techniques will eventually lead to an
effective bound that will only depend on the genus and the bad
reduction types.

The following is a trivial extension of Lemma
2 in~\cite{Col85a}.
\begin{lemma}\label{keyinteglem}
  Let $f(t)=\sum a_n t^n$  be a power series with coefficients in
  $\Cp$. Let $\beta$ be the function whose graph is the bottom of the
  Newton polygon for $f$, and let $k\ge 0$ be an integer. Then, for any
  $s>0$ we have
  \begin{equation*}
    \# \{z\in \Cp,\; v(z)\ge s,\; f(z)=0\} \le \max \{n\ge k,\; v(a_k) - s
    (n-k) \ge \beta(n)\},
  \end{equation*}
  where $v$ is the $p$-adic valuation normalized so that $v(p)=1$.
\end{lemma}
\begin{proof}
  Suppose that there
  are $m$ zeros of $f$ with valuation $\ge s$. Since $\beta(k) \le
  v(a_k)$ it suffices to show that
  $\beta(m) \ge \beta(k)-s(m-k)$. But this is clear, as by assumption, by
  the properties of the Newton polygon, the slopes of $\beta$ up to the
  point where $x=m$ are at most $-s$.
\end{proof}
Obviously, in the above lemma, any lower bound for $\beta$ may be used
instead of $\beta$ to obtain an upper bound on the number of
zeros. To use this in our setting, we need the following obvious
lemma.
\begin{lemma}\label{logbounds}
  Suppose, under the assumptions of the above lemma, that $f$ is an
  $r$-fold iterated integral of forms $\omega_1,\ldots,\omega_r$ which have
  integral coefficients, and that furthermore the constants of
  integration at each integration are integral as well. Then we have
  the lower bound $\beta(k)\ge  -r\log_p(k)$.
\end{lemma}
In~\cite{Col85a}, Coleman uses the following corollary to get an upper
bound on the number of rational points.
\begin{corollary}
  Suppose under the assumptions of Lemma~\ref{keyinteglem} that $f'$ has
  integral coefficients and that its reduction has order $k$. Then, 
  \begin{equation*}
    \# \{z\in \Cp,\; v(z)\ge s,\; f(z)=0\} \le \max \{n\ge k,\;  s
    (n-k) \le [\log_p(n)]\}\;.
  \end{equation*}
\end{corollary}
Coleman then uses a bound on the sum of all possible $k$'s in the
above corollary, over all power series occurring as local power series
expansions of a Coleman integral, to obtain a global bound on the
number of rational points, when Chabauty's method applies, which
depends only on the genus and the number of points in the reduction of
the curve.

As an application of our techniques, we give a fairly easily computed
bound on the number of integral points. The reader will recognize that
one may improve the bound and remove restrictions at the cost of
making the bound depend on more complicated data.

Specializing  Lemmas \ref{keyinteglem}~and~\ref{logbounds} to the case
$r=2$, $k=s=1$, we get that for power series satisfying the conditions
of Lemma~\ref{logbounds}, the number of $p\Zp$ roots is bounded from
above by
\begin{equation}\label{upsilon}
  \Upsilon(v(a_1))\;,\text{ with } \Upsilon(m) = \max \{n\ge 1,\;
  m - (n-1) \ge -2\log_p(n)\}\;.
\end{equation}
We note, for ease of application that
\begin{enumerate}
\item if $m+3<p$, then $\Upsilon(m)= m+2$.
\item if $m+5<p^2$, then $\Upsilon(m)= m+4$.
\end{enumerate}

\begin{theorem}\label{boundthm}
  Suppose, under the assumptions of Theorem~\ref{cutting-function},
  that the polynomial $f$ has no roots modulo $p$.  Let $\alpha_{ij}$
  be the constants appearing in the theorem and let $\Upsilon$ be the
  function from~\eqref{upsilon}.
  
  Let $P_m$ be $\Zp$-points of $X$ lifting all $\Z/p$ points of the
  special fiber save $\infty$, and for each $m$, let $t_m$ be a local parameter at
  $P_m$, whose reduction modulo $p$ is also a parameter. Define
  constants $\beta_{im}$, $\gamma_{im}$, $\gambar_{im}$ by
  \begin{equation*}
    \beta_{im} = \frac{\omega_i}{dt_m}(P_m)
  \end{equation*}
  and
  \begin{equation*}
    \gamma_{im}= \int_{\cansec_0}^{P_m} \omega_i\;,\; 
    \gambar_{im}= \int_{\cansec_0}^{P_m} \bom_i\;.
  \end{equation*}
  Let $L = \max\{0, -v(\alpha_{ij}), 0\le i,j\le g-1\} $ and let
  \begin{equation*}
    \delta_m = 2\sum_{i=0}^{g-1} \beta_{im}\gambar_{im} + \sum_{i\le j}
    \alpha_{ij} (\beta_{im} \gamma_{jm}+\beta_{jm} \gamma_{im}).
  \end{equation*}
  Then, the number of integral solutions to $y^2=f(x)$ is bounded from
  above by
  \begin{equation*}
    (1+\prod_q (m_q-1)) \sum_m \Upsilon(v_p(\delta_m) + L)
  \end{equation*}
  where, as in Remark~\ref{comprem}, $m_q$ denotes the number of
  multiplicity $1$ components at the fiber above $q$.
\end{theorem}
\begin{proof}
  Note that the combined assumption that we are looking at integral,
  rather than $p$-integral points, and that $f$ has no roots modulo $p$,
  means that we only need to consider non-Weierstrass residue
  disks. Multiplying by $p^L$ clears denominators from the
  $\alpha_{ij}$ and it is therefore evident that the expansion of $p^L
  \rho$ around $P_m$ satisfies 
  the condition of Lemma~\ref{logbounds}, except possibly that the constant
  coefficient of the power series is not integral. The first coefficient
  is exactly $p^L \delta_m$. Thus, to conclude the proof by Theorem~\ref{cutting-function},
  Lemma~\ref{keyinteglem}, Lemma~\ref{logbounds}, and Remark~\ref{comprem}, it
  suffices to note that having a non-integral value for the constant
  coefficient of the power series flattens the Newton polygon and so
  improves the bound on the number of solutions.
\end{proof}
\begin{remark}\label{boundingredients}
  To use the theorem above to compute a bound on the number of integral
  points, one needs the following:
  \begin{itemize}
  \item Computation of (non-iterated) Coleman integrals. These are required for
    the computation of $\gamma_{ij}$ , $\gambar_{ij}$, and $\alpha_{ij}$.
  \item A basis of $J(\Q)\otimes\Q$.
  \item Computation of $p$-adic height pairings~\cite{Bes-Bal10,BaMuSt12}.
  \end{itemize}
  In particular, explicit iterated Coleman integrals are not required. One may use them,
  however, or the techniques of Section~\ref{sec:examples}, to get the values of
  $\tau$ at the points $P_m$, and these may be used to give a more
  precise bound, typically with more work.
\end{remark}

At the moment, we are unfortunately unable to produce a
Coleman-Chabauty-like bound, because we are unable to control a
quantity like the sum of all $k$'s in Coleman's case for iterated
Coleman integrals. We plan to return to this question in future work.

\begin{remark}
  We may observe that we can treat primes in the denominators
  of our points, at the cost of extending the set $T$, as long as we
  have a bound on the denominator.
\end{remark}

\section{Examples}\label{examples}
Here we give some examples illustrating the techniques described in this paper.
\subsection{Genus 1}
Let $X: y^2 = x^3 - 3024x + 70416$, which has minimal model with Cremona label
``57a1''.  We have that $X(\Q)$ has Mordell-Weil rank 1, with $P = (60,-324)$ a generator of the free part of the Mordell-Weil group. We take as our working prime $p = 7$. 
Since the given equation of $X$ is not a minimal Weierstrass equation,
Remark~\ref{g1minimal} does not apply. 
The only primes $q$ where the Zariski closure of $X$ over $\Spec(\Z_q)$
is not already regular are $q=2,\,3$.
Using {\tt Magma} \cite{magma}, we compute a desingularization $\XX$ of the
Zariski closure of $X$ in the strong sense.
The special fibers are of the form 
\[
\XX_2 = \Gamma_{2,0} + \Gamma_{2,1} + 2\Gamma_{2,2}
\]
and 
\[
\XX_3 = \Gamma_{3,0} + \Gamma_{3,1} + \Gamma_{3,2} + 2\Gamma_{3,3},
\]
where all components $\Gamma_{q,i}$ have genus~0 and $\Gamma_{q,0}$ is the
component which $\infty$ intersects.
The intersection matrices are given by:
\begin{equation*}
  M_2=\left(\begin{array}{rrr}
      -2&  0 & 1\\
      0& -2 & 1\\
      1 & 1& -1\\
    \end{array}\right)\;,
  \quad
  M_3=\left(\begin{array}{rrrr}
      -2 & 0&  0&  1\\
      0 &-4&  2&  1\\
      0  &2& -2 & 0\\
      1 & 1&  0& -1\\
    \end{array}\right)\;.
\end{equation*}
Writing $\div(\omega) = V_2 + V_3$, we have 
\begin{equation} \label{Vg1}
  V_2 = -\Gamma_{2,1} + a_{2,2}\Gamma_{2,2}\quad\textrm{ and }\quad V_3 = -\Gamma_{3,1} -
  \Gamma_{3,2} + a_{3,3}\Gamma_{3,3},
\end{equation}
computed using the method alluded to in Section~\ref{explicit-intersection}.
Here $a_{2,2}$ and $a_{3,3}$ are irrelevant, since no $\Z_q$-section can
intersect a component of multiplicity at least~2.
\subsubsection{Computing local and global $p$-adic heights of a Mordell-Weil generator}
Using {\tt Sage} \cite{sage}, we compute the local height above $7$ using Coleman integration and cotangent vectors: \[h_7(P-wP,P-wP) = 4 \cdot 7 + 4 \cdot 7^{2} + 7^{3} + 7^{4} + 6 \cdot 7^{5} + 6 \cdot 7^{6} + O(7^{7}),\]
and we can compute the contributions away from $p=7$ using~\eqref{Vg1}
and~\eqref{phi_formula} to deduce the global $7$-adic height:
\begin{align*}h(P-wP,P-wP) &= h_7(P-wP,P-wP) -2\log_7(2) - 2\log_7(3) \\
  &=6 \cdot 7 + 5 \cdot 7^{2} + 4 \cdot 7^{3} + 4 \cdot 7^{5} + 7^{6} + O(7^{7}),\end{align*}
which one can check agrees with Harvey's implementation \cite{harvey:heights} of Mazur-Stein-Tate's algorithm \cite{Maz-Ste-Tat05} for $h(2P)$.

This gives \begin{align*}\tau(P) &= \frac{1}{4}\log_7(4y(P)^2) + \frac{1}{4}h(P-wP,P-wP) \\
  &= 4 \cdot 7^{2} + 4 \cdot 7^{3} + 5 \cdot 7^{4} + 3 \cdot 7^{5} + 6 \cdot 7^{6} + O(7^{7}).\end{align*}

Since $\tau(P)$  is the local component of the $7$-adic height above $7$ (that is, by definition, we have $\tau(P) = h_7(P-\infty,P-\infty)$), we can put this together with the contributions away from $p=7$ to compute the global $7$-adic height:
\begin{align*}h(P - \infty) = h(P-\infty,P-\infty) &= \tau(P) - 2\log_7(2) - \frac{5}{2}\log_7(3)\\
  &= 5 \cdot 7 + 6 \cdot 7^{2} + 2 \cdot 7^{3} + 5 \cdot 7^{4} + 2 \cdot 7^{5} + O(7^{7}).\end{align*}
\subsubsection{Computing $\alpha_{00}$}
We have \[\alpha_{00} = \frac{h(P - \infty)}{\left(\int_{\infty}^P \omega_0\right)^2}.\] This gives 
\begin{align*}\rho(P) &= \tau(P) - \alpha_{00}\left(\int_{\infty}^P \omega_0\right)^2\\
  &= 2 \cdot 7 + 4 \cdot 7^{2} + 7^{3} + 7^{5} + 6 \cdot 7^{6} + O(7^{7}).
\end{align*}
\subsubsection{$\rho(z)$ values}
By~\eqref{rhoequation}, we can combine~\eqref{phi_formula}
and~\eqref{Vg1} to compute the set $T$ of all possible values of $\rho(z)$ for
$7$-integral points $z$ on $X$.
It turns out that we have 
\[
T = \{i\cdot\log_7(2) + j\cdot\log_7(3)\;:\;i = 0,\,2,\;
j = 0,\, 2,\, 5/2\}\;.
\]
We compute the  values of \[\rho(z) = \tau(z) -\alpha_{00}\left(\int_{\infty}^z \omega_0\right)^2\] for the sixteen integral points:
\begin{align*}&(-48, \pm 324), (-12, \pm 324), (24, \pm 108),
  (33, \pm 81),  \\
  &(40, \pm 116), (60, \pm 324), (132, \pm 1404), (384, \pm 7452),\end{align*}
and find that they all lie in the set $T$, as summarized below.

\begin{center}
  \begin{tabular}{|| c | c ||}
    \hline
    $z$ & $\rho(z)$ \\
    \hline
        $(-48,\pm324)$ & $2\log_7(2) +\frac{5}{2}\log_7(3) $ \\
    $(-12,\pm 324)$ & $2\log_7(2) + 2\log_7(3)$  \\
    $(24, \pm108)$ & $2\log_7(2) + 2\log_7(3)$  \\
    $(33,\pm 81)$ & $\frac{5}{2}\log_7(3)$\\
    $(40,\pm 116)$ &   $2\log_7(2)$\\
    $(60,\pm 324)$ &   $2\log_7(2) + \frac{5}{2}\log_7(3)$ \\
    $(132, \pm1404)$ &  $2\log_7(2) + 2\log_7(3)$ \\
    $(384,\pm 7452)$ & $2\log_7(2) + \frac{5}{2}\log_7(3)$ \\    
    \hline
  \end{tabular}
\end{center}

\begin{remark}Note a similar result is obtained in \cite{BCKW12} for a semistable rank 1 elliptic curve, up to normalization. Our $\rho(z)$ values, calculated on a minimal model, are twice the corresponding values of $||w||$.
The reason is that our normalization of the $p$-adic height corresponds to the
normalization in~\cite{harvey:heights}, which is $2p$ times the normalization
in~\cite{Maz-Ste-Tat05}.
One the other hand, the normalization used in \cite{BCKW12} is $p$ times the normalization
in~\cite{Maz-Ste-Tat05}.
One can also check directly that our  formulas for the local contributions are
twice the formulas for the local contributions in \cite{BCKW12}.
\end{remark}

\subsection{Genus 2}\label{ex:g2}
Let $X: y^2 = f(x)$, where $f(x) = x^3(x-1)^2 + 1$.
Using a~2-descent as implemented in {\tt Magma}, one checks that the Jacobian
$J$ of $X$ has Mordell-Weil rank~2 over $\Q$.
We are grateful to Michael Stoll for checking that $X$ has only the six obvious
integral points  $(2,\pm 3), (1,\pm 1), (0,\pm 1)$, using the
methods of~\cite{BMSST08}. Let $P = (2,-3), Q = (1,-1), R = (0,1) \in X(\Q)$.
Then $[(P)-(\infty)]$ and $[(Q)-(R)]$ are generators of the free part of the
Mordell-Weil group. We take as our working prime $p = 11$. 

The only prime $q$ where the Zariski closure of $X$ over $\Spec(\Z_q)$
is not already regular is $q=2$.
Using {\tt Magma}, we find that there is a desingularization of the
Zariski closure of $X$ in the strong sense such that 
\[
\XX_2 = \Gamma_0+\Gamma_1+\Gamma_2+\Gamma_3,
\]
where all components $\Gamma_i$ have genus~0.
Since the intersection matrix is
\begin{equation*}
  M=\left(\begin{array}{rrrr}
      -4&  2&  1&  1\\
      2& -2&  0&  0\\
      1&  0& -2&  1\\
      1&  0 & 1& -2\\
    \end{array}\right)\,,\end{equation*}
this implies that
$\div(\omega)=2\infty$ has no vertical component. 
Note that $\Gamma_0$ and $\Gamma_1$ have an intersection point of multiplicity~2 and
$\Gamma_0, \Gamma_2, \Gamma_3$ intersect in one point, so $\XX$ is not
semistable.

\subsubsection{Heights of Mordell-Weil generators} Let $D_1 = (P) - (\infty)$
and $D_2 =(Q) - (R)$.
Using {\tt Sage}, we compute the local $11$-adic height above $11$ as 
\begin{align*}h_{11}(D_1, D_1) &= \frac{1}{4}\log_{11}(4y(P)^2) +
\frac{1}{4}h_{11}((P) - (wP), (P) - (wP))\\
  &= 10 \cdot 11 + 2 \cdot 11^{2} + 9 \cdot 11^{3} + 9 \cdot 11^{4} + 8 \cdot 11^{5} + O(11^{7}).\end{align*}
Away from $p = 11$, $h(D_1, D_1)$ only has a contribution  at $2$ (computed
using~\eqref{phi_formula}), and we have 
\begin{align*}h(D_1,D_1) &= -\frac{2}{3}\log_{11}(2) + h_{11}(D_1, D_1) \\
  &= 7 \cdot 11 + 11^{2} + 8 \cdot 11^{3} + 8 \cdot 11^{4} + 7 \cdot 11^{5} + 2 \cdot 11^{6} + O(11^{7}).\end{align*}

Using $2(P) - (S) - (wS) \sim 2(P) - 2(\infty)$, where  $S = (-2/9, -723/9^3)$, we compute
\begin{align*}
  h_{11}(D_1, D_2) &= \frac{1}{2} \left(h_{11}((P)-(S), (Q)-(R)) -
    h_{11}((P)-(wS), (wQ)-(wR))\right)  \\ 
  &= 4 \cdot 11 + 4 \cdot 11^{2} + 8 \cdot 11^{3} + 4 \cdot 11^{4} + 9 \cdot 11^{5} + O(11^{7})
\end{align*}
and
\begin{align*}
  h(D_1, D_2) &= \frac{1}{2}\left(h((P)-(S),(Q)-(R)) - h((P)-(wS),(wQ)-(wR))\right)\\
  &=\frac{1}{2}\left(h_{11}((P)-(S), (Q)-(R)) - \left(-\frac{1}{3}\log_{11}(2) +
    h_{11}((P)-(wS), (wQ)-(wR))\right)\right)  \\
  &= 3 \cdot 11 + 10 \cdot 11^{2} + 9 \cdot 11^{3} + 6 \cdot 11^{4} + 4 \cdot 11^{5} + 4 \cdot 11^{6} + O(11^{7}).
\end{align*}
Finally, we have
\begin{align*}
  h_{11}(D_2, D_2) &= h_{11}((Q)-(R), (wR)-(wQ))\\
  &=   3 \cdot 11 + 4 \cdot 11^{2} + 4 \cdot 11^{3} + 11^{4} + 11^{5} + 3 \cdot 11^{6} + O(11^{7})
\end{align*} and
\begin{align*}
  h(D_2, D_2) &= \frac{5}{6}\log_{11}(2) + h_{11}(D_2,D_2)  \\
  &= 8 \cdot 11 + 2 \cdot 11^{4} + 6 \cdot 11^{5} + 6 \cdot 11^{6} +  O(11^{7}).
\end{align*}
Note that for $h(D_1, D_2)$ and $h(D_2, D_2)$ our chosen representatives are
disjoint, so we can use the techniques of~\cite{Bes-Bal10,BaMuSt12}.

\subsubsection{The coefficients $\alpha_{ij}$}

We compute the coefficients $\alpha_{ij}$ using the matrix of the various products of Coleman integrals evaluated at $D_1, D_2$ (as in~\cite{BBK09}) and the global $11$-adic heights computed above:

\begin{equation*}
\fontsize{8}{4}{
  \left(\begin{array}{r}
      \alpha_{00}\\
      \alpha_{01}\\
      \alpha_{11}
    \end{array}\right) = 
  \left(\begin{array}{ccc}
      \int_{D_1}\omega_0\int_{D_1}\omega_0  & \int_{D_1}\omega_0\int_{D_1}\omega_1 & \int_{D_1}\omega_1\int_{D_1}\omega_1\\
      \int_{D_1}\omega_0\int_{D_2}\omega_0  &  \frac{1}{2}\left(\int_{D_1}\omega_0\int_{D_2}\omega_1 +\int_{D_1}\omega_1\int_{D_2}\omega_0 \right) & \int_{D_1}\omega_1\int_{D_2}\omega_1\\
      \int_{D_2}\omega_0\int_{D_2}\omega_0  & \int_{D_2}\omega_0\int_{D_2}\omega_1 & \int_{D_2}\omega_1\int_{D_2}\omega_1\\
    \end{array}\right)^{-1}\left(\begin{array}{r}
      h(D_1,D_1)\\
      h(D_1,D_2)\\
      h(D_2,D_2)
    \end{array}\right).}
\end{equation*}

This gives
\begin{align*}
\alpha_{00} &=8 \cdot 11^{-1} + 10 + 11 + 10 \cdot 11^{2} + 6 \cdot 11^{3} + 3 \cdot 11^{4} + 6 \cdot 11^{5} + 4 \cdot 11^{6} + 6 \cdot 11^{7} +  O(11^{8})\\
\alpha_{01} &= 11^{-1} + 5 + 4 \cdot 11 + 10 \cdot 11^{2} + 7 \cdot 11^{3} + 5 \cdot 11^{4} + 8 \cdot 11^{5} + 8 \cdot 11^{6} + 5 \cdot 11^{7} +O(11^{8})\\
\alpha_{11} &= 5 \cdot 11^{-1} + 10 + 6 \cdot 11 + 4 \cdot 11^{2} + 8 \cdot 11^{3} + 5 \cdot 11^{4} + 9 \cdot 11^{5} + 6 \cdot 11^{6} + 11^{7} + O(11^{8}).\end{align*}
\subsubsection{Values of $\rho$ on 11-integral points}
Since $\div(\omega)$ is horizontal and the polynomial $f$ is monic, we only need~\eqref{phi_formula} to compute
the set 
\[
    T =\left\{0, \frac{1}{2}\log_{11}(2), \frac{2}{3}\log_{11}(2)\right\}
\]
of all possible values of $\rho$ on 11-integral points.

\subsubsection{Constructing the dual basis}Now we describe how to
construct power series whose zeros contain the set of integral points
of $X$. For this, we will need to use the expression for $\tau$ given
in Theorem~\ref{tauthm}. (Note that thus far, we have only used the expression for $\tau$ as a local
height, which does not give a power series expansion.)

First, we must compute the dual basis for $\{\omega_0,\omega_1\}$ in $W$. To make the
calculations compatible with previous height computations~\cite{Bes-Bal10}, we take $W$ to be 
the unit root subspace. 

Recall that a basis for $W$ is given by $\{(\phi^*)^n\omega_2,(\phi^*)^n\omega_3\}$, where $\phi$ is a lift of $p$-power Frobenius and $n$ is the working precision.  Thus we let $\tilde{\omega}_j$,
for $j=2,3$, be the projection of $\omega_j$ on the unit root subspace
along the space of holomorphic forms. We clearly have
$[\tilde{\omega}_j] \cup [\omega_i]= [\omega_j] \cup [\omega_i]$ for $i\le 1 < j$.

Here is the cup product matrix for $X$:
\[\left(\begin{array}{rrrr}
0 & 0 & 0 & \frac{1}{3} \\
0 & 0 & 1 & \frac{4}{3} \\
0 & -1 & 0 & \frac{1}{3} \\
-\frac{1}{3} & -\frac{4}{3} & -\frac{1}{3} & 0
\end{array}\right),\]
where the $ij$th entry is giving the value of $[\omega_i] \cup
[\omega_j]$. The inverse of the bottom left corner (which gives the
above cup product values $[\tilde{\omega}_j] \cup [\omega_i]$) is
\[\left(\begin{array}{rr}
 4 & -3 \\
 -1 & 0
\end{array}\right),\] and this immediately gives
\begin{align*}
\bom_0&= 4\tilde{\omega}_2  - 3\tilde{\omega}_3\\
\bom_1&=  -\tilde{\omega}_2
\end{align*}
(note that this is consistent with~\eqref{eq:bom}).

We now need to compute the projection of $\omega_2, \omega_3$ with respect to the basis $\{\omega_0,\omega_1,\tilde{\omega}_2,\tilde{\omega}_3\}$.

We do this by inverting the matrix with first two columns from the identity matrix and last two columns from the Frobenius matrix raised to the working precision. Calling the resulting  upper-right submatrix 
\[\left(\begin{array}{rr}
a & b \\
c & d
\end{array}\right),\]
we find
\begin{align*}
a &=3 + 10 \cdot 11 + 10 \cdot 11^{2} + 11^{4} + 11^{5} + 5 \cdot 11^{6} + 11^{7} + 3 \cdot 11^{8} + 4 \cdot 11^{9} + 5 \cdot 11^{10} + O(11^{11})\\
b &= 6 + 8 \cdot 11 + 4 \cdot 11^{2} + 11^{3} + 7 \cdot 11^{4} + 9 \cdot 11^{5} + 2 \cdot 11^{6} + 6 \cdot 11^{7} + 2 \cdot 11^{8} + 5 \cdot 11^{9} + 6 \cdot 11^{10}  + O(11^{11})\\
c &=  4 + 3 \cdot 11 + 6 \cdot 11^{2} + 6 \cdot 11^{3} + 9 \cdot 11^{4} + 10 \cdot 11^{5} + 4 \cdot 11^{6} + 5 \cdot 11^{7} + 2 \cdot 11^{8} + 2 \cdot 11^{9}  + O(11^{11})\\
d &=  6 + 11^{2} + 9 \cdot 11^{3} + 5 \cdot 11^{4} + 7 \cdot 11^{5} + 4 \cdot 11^{6} + 8 \cdot 11^{8} + 2 \cdot 11^{10} + O(11^{11})
\end{align*} so that we have 
\begin{align*}
\tilde{\omega_2} &= \omega_2 - a\omega_0 - c\omega_1\\
\tilde{\omega_3} &= \omega_3 - b\omega_0 - d\omega_1
\end{align*}
which gives 
\begin{align*}
\bar{\omega_0} &= (-4a+3b)\omega_0 + (-4c + 3d)\omega_1 + 4\omega_2 - 3\omega_3\\
\bar{\omega_1} &= a \omega_0 + c\omega_1 - \omega_2.
\end{align*}
One can check that  this  gives $[\bom_i] \cup [\omega_j] = \delta_{ij}$.

Now we use the dual basis to construct power series within each of the following $\F_{11}$-residue disks:
\[\{ \overline{(0, \pm 1)}, \overline{(1, \pm 1)}, \overline{(2, \pm 3)}, \overline{(8, \pm 3)}, \overline{(4, \pm 4)}, \overline{(6,0)} \}\;.\]

For example, starting with the residue disk of $\overline{(0,1)}$, we first lift $\overline{(0,1)}$ to the point $P = (0, 1)$. We compute $\tau(P)$ using its interpretation as the $p$-component of the global $p$-adic height pairing of $P - \infty$: \[\tau(P) = 3 \cdot 11 + 2 \cdot 11^{2} + 11^{3} + 4 \cdot 11^{4} + 2 \cdot 11^{5} + 9 \cdot 11^{6} + 8 \cdot 11^{7} + 6 \cdot 11^{8} + 9 \cdot 11^{9} + 7 \cdot 11^{10} + O(11^{11}).\] Using this, we write down a power series expansion for $\rho$ in the disk of $(0,1)$:
\[\rho(z) = \tau(P) -2\sum_{i=0}^{1} \left( \int_{P}^{z} \omega_i  \bom_i  + \int_{P}^{z} \omega_i  \int_{\cansec_0}^{z} \bom_i \right)  - \sum_{0 = i\leq j}^1 \alpha_{ij}\int_{\cansec_0}^z \omega_i\int_{\cansec_0}^z \omega_j.\]
Setting this equal to each of the values in the set $T = \{0,  \frac{1}{2}\log_{11}(2), \frac{2}{3}\log_{11}(2)\}$, we find that the points $z$ in the residue disk $\overline{(0,1)}$ with $x$-coordinates as indicated in the table below 
are the only $\Z_{11}$-points $z$ which achieve $\rho(z)$ values in the set $T$. 

This computation recovers the integral point $(0,1)$ and tells us that no other integral points in the disk exist with $x$-coordinate having absolute value less than $11^{11}$.   Here we summarize the computation in all of the residue disks:
\begin{center}
 \begin{tabular}{||c |  r  | c ||}
    \hline
  disk & $x(z)$ & $\rho(z)$ \\
  \hline
& $11 + 4 \cdot 11^{3} + 9 \cdot 11^{4} + 3 \cdot 11^{5} + 8 \cdot 11^{6} +O(11^7)$ & $0$   \\
& $9 \cdot 11 + 5 \cdot 11^{2} + 8 \cdot 11^{3} + 4 \cdot 11^{4} + 8 \cdot 11^{5} + 8 \cdot 11^{6} +O(11^7)$ & $0$ \\
$\overline{(0,\pm 1)}$ & $2 \cdot 11 + 3 \cdot 11^{2} + 2 \cdot 11^{3} + 7 \cdot 11^{4} + 9 \cdot 11^{5} + 9 \cdot 11^{6} + O(11^7)$ & $\frac{1}{2}\log_{11}(2)$\\
& $8 \cdot 11 + 10 \cdot 11^{2} + 10 \cdot 11^{3} + 2 \cdot 11^{4} + 11^{5} + 10 \cdot 11^{6} + O(11^7)$ & $\frac{1}{2}\log_{11}(2)$\\
& $O(11^7)$ &  $\frac{2}{3}\log_{11}(2)$ \\
& $10 \cdot 11 + 11^{2} + 6 \cdot 11^{3} + 6 \cdot 11^{4} + 8 \cdot 11^{5} + 8 \cdot 11^{6} +O(11^7)$ &  $\frac{2}{3}\log_{11}(2)$ \\

\hline  

 & $1 +O(11^{7})$&  $\frac{1}{2}\log_{11}(2)$\\ 
$\overline{(1, \pm 1)}$&  $1 + 5 \cdot 11 + 9 \cdot 11^{2} + 8 \cdot 11^{3} + 8 \cdot 11^{4} + 3 \cdot 11^{5} + 4 \cdot 11^{6} +O(11^7)$ & $\frac{1}{2}\log_{11}(2)$\\
&  $1 + 6 \cdot 11 + 10 \cdot 11^{2} + 4 \cdot 11^{4} + 10 \cdot 11^{5} + 3 \cdot 11^{6} + O(11^7)$ &  $\frac{2}{3}\log_{11}(2)$\\
& $1 + 10 \cdot 11 + 9 \cdot 11^{2} + 5 \cdot 11^{3} + 10 \cdot 11^{4} + 8 \cdot 11^{5} + O(11^7)$ & $\frac{2}{3}\log_{11}(2)$\\
\hline
& $2 + 6 \cdot 11 + 6 \cdot 11^{2} + 11^{3} + 10 \cdot 11^{4} + 3 \cdot 11^{5} + 2 \cdot 11^{6} + O(11^7)$ & $0$  \\
& $2 + 8 \cdot 11 + 8 \cdot 11^{2} + 7 \cdot 11^{3} + 6 \cdot 11^{4} + 8 \cdot 11^{5} + O(11^7)$ & $0$ \\
$\overline{(2, \pm 3)}$ & $2 + 5 \cdot 11 + 5 \cdot 11^{2} + 5 \cdot 11^{3} + 4 \cdot 11^{4} + 5 \cdot 11^{5} + 3 \cdot 11^{6} + O(11^7)$ & $\frac{1}{2}\log_{11}(2)$ \\
& $2 + 9 \cdot 11 + 4 \cdot 11^{2} + 5 \cdot 11^{3} + 9 \cdot 11^{4} + 2 \cdot 11^{6} +O(11^7)$ & $\frac{1}{2}\log_{11}(2)$ \\
& $2 + O(11^7)$ & $\frac{2}{3}\log_{11}(2)$ \\
& $2 + 3 \cdot 11 + 11^{2} + 4 \cdot 11^{3} + 11^{5} + 7 \cdot 11^{6} +O(11^7)$ &  $\frac{2}{3}\log_{11}(2)$  \\
\hline
$\overline{(8,\pm 3)}$ &  $8 + 7 \cdot 11 + 9 \cdot 11^{2} + 4 \cdot 11^{3} + 7 \cdot 11^{4} + 2 \cdot 11^{5} + 3 \cdot 11^{6} +O(11^7)$ &  $0$ \\
& $8 + 2 \cdot 11^{2} + 2 \cdot 11^{3} + 4 \cdot 11^{4} + 7 \cdot 11^{5} + 4 \cdot 11^{6} +O(11^7)$ & $0$  \\
\hline
$\overline{(4, \pm 4)}$ & $-$ & $-$ \\
\hline
$\overline{(6,0)}$ &  $6 + 9 \cdot 11 + 9 \cdot 11^{2} + 6 \cdot 11^{3} + 6 \cdot 11^{4} + 4 \cdot 11^{5} + 2 \cdot 11^{6} +  O(11^{7})$  & $\frac{1}{2}\log_{11}(2)$  \\
\hline
\end{tabular}
\end{center}

In particular, here are the recovered integral points and their corresponding $\rho$ values:
\begin{center}
 \begin{tabular}{|| c | c ||}
    \hline
  $z$ &  $\rho(z)$ \\
    \hline
$(2, \pm3)$  & $\frac{2}{3}\log_{11}(2)$ \\
$(1, \pm1)$ & $\frac{1}{2}\log_{11}(2)$\\
$(0,\pm1)$ & $\frac{2}{3}\log_{11}(2)$ \\

\hline
\end{tabular}
\end{center}
In forthcoming work we verify that these are the only
integral points by combining the methods of the present paper with a suitable variant of the
Mordell-Weil sieve.

\bibliography{total}
\bibliographystyle{amsplain}
\end{document}